\documentclass[12pt]{amsart}
\usepackage{amsmath,amsthm,amsfonts,amssymb,latexsym,enumerate,url,hyperref,color} 

\begin{document}

\theoremstyle{plain}

\newtheorem{thm}{Theorem}[section]
\newtheorem{lem}[thm]{Lemma}
\newtheorem{conj}[thm]{Conjecture}
\newtheorem{pro}[thm]{Proposition}
\newtheorem{cor}[thm]{Corollary}
\newtheorem{que}[thm]{Question}
\newtheorem{rem}[thm]{Remark}
\newtheorem{defi}[thm]{Definition}

\newtheorem*{thmA}{THEOREM A}
\newtheorem*{thmB}{THEOREM B}
\newtheorem*{corC}{COROLLARY C}
\newtheorem*{conjC}{CONJECTURE C}
\newtheorem*{conjD}{CONJECTURE D}
\newtheorem*{thmE}{THEOREM E}

\newtheorem*{thmAcl}{Main Theorem$^{*}$}
\newtheorem*{thmBcl}{Theorem B$^{*}$}

\numberwithin{equation}{section}

\newcommand{\Maxn}{\operatorname{Max_{\textbf{N}}}}
\newcommand{\Syl}{\operatorname{Syl}}
\newcommand{\dl}{\operatorname{\mathfrak{d}}}
\newcommand{\Con}{\operatorname{Con}}
\newcommand{\cl}{\operatorname{cl}}
\newcommand{\Stab}{\operatorname{Stab}}
\newcommand{\Aut}{\operatorname{Aut}}
\newcommand{\Ker}{\operatorname{Ker}}
\newcommand{\IBr}{\operatorname{IBr}}
\newcommand{\Irr}{\operatorname{Irr}}
\newcommand{\SL}{\operatorname{SL}}
\newcommand{\FF}{\mathbb{F}}
\newcommand{\NN}{\mathbb{N}}
\newcommand{\N}{\mathbf{N}}
\newcommand{\C}{\mathbf{C}}
\newcommand{\OO}{\mathbf{O}}
\newcommand{\F}{\mathbf{F}}

\renewcommand{\labelenumi}{\upshape (\roman{enumi})}

\newcommand{\GL}{\operatorname{GL}}
\newcommand{\Sp}{\operatorname{Sp}}
\newcommand{\PGL}{\operatorname{PGL}}
\newcommand{\PSL}{\operatorname{PSL}}
\newcommand{\SU}{\operatorname{SU}}
\newcommand{\PSU}{\operatorname{PSU}}
\newcommand{\PSp}{\operatorname{PSp}}

\providecommand{\V}{\mathrm{V}}
\providecommand{\E}{\mathrm{E}}
\providecommand{\ir}{\mathrm{Irr_{rv}}}
\providecommand{\Irrr}{\mathrm{Irr_{rv}}}
\providecommand{\re}{\mathrm{Re}}

\def\irrp#1{{\rm Irr}_{p'}(#1)}

\def\Z{{\mathbb Z}}
\def\C{{\mathbb C}}
\def\Q{{\mathbb Q}}
\def\irr#1{{\rm Irr}(#1)}
\def\ibr#1{{\rm IBr}(#1)}
\def\irrv#1{{\rm Irr}_{\rm rv}(#1)}

\def \c#1{{\cal #1}}
\def\cent#1#2{{\bf C}_{#1}(#2)}
\def\syl#1#2{{\rm Syl}_#1(#2)}
\def\nor{\triangleleft\,}
\def\oh#1#2{{\bf O}_{#1}(#2)}
\def\Oh#1#2{{\bf O}^{#1}(#2)}
\def\zent#1{{\bf Z}(#1)}
\def\det#1{{\rm det}(#1)}
\def\ker#1{{\rm ker}(#1)}
\def\norm#1#2{{\bf N}_{#1}(#2)}
\def\alt#1{{\rm Alt}(#1)}
\def\iitem#1{\goodbreak\par\noindent{\bf #1}}
   \def \mod#1{\, {\rm mod} \, #1 \, }
\def\sbs{\subseteq}

\def\gc{{\bf GC}}
\def\ngc{{non-{\bf GC}}}
\def\ngcs{{non-{\bf GC}$^*$}}
\newcommand{\notd}{{\!\not{|}}}
\newcommand{\Out}{{\mathrm {Out}}}
\newcommand{\Mult}{{\mathrm {Mult}}}
\newcommand{\Inn}{{\mathrm {Inn}}}
\newcommand{\IBR}{{\mathrm {IBr}}}
\newcommand{\IBRL}{{\mathrm {IBr}}_{\ell}}
\newcommand{\IBRP}{{\mathrm {IBr}}_{p}}
\newcommand{\ord}{{\mathrm {ord}}}
\def\id{\mathop{\mathrm{ id}}\nolimits}
\renewcommand{\Im}{{\mathrm {Im}}}
\newcommand{\Ind}{{\mathrm {Ind}}}
\newcommand{\diag}{{\mathrm {diag}}}
\newcommand{\soc}{{\mathrm {soc}}}
\newcommand{\End}{{\mathrm {End}}}
\newcommand{\sol}{{\mathrm {sol}}}
\newcommand{\Hom}{{\mathrm {Hom}}}
\newcommand{\Mor}{{\mathrm {Mor}}}
\newcommand{\St}{{\sf {St}}}
\def\rank{\mathop{\mathrm{ rank}}\nolimits}
\newcommand{\Tr}{{\mathrm {Tr}}}
\newcommand{\tr}{{\mathrm {tr}}}
\newcommand{\Gal}{{\it Gal}}
\newcommand{\Spec}{{\mathrm {Spec}}}
\newcommand{\ad}{{\mathrm {ad}}}
\newcommand{\Sym}{{\mathrm {Sym}}}
\newcommand{\Char}{{\mathrm {char}}}
\newcommand{\pr}{{\mathrm {pr}}}
\newcommand{\rad}{{\mathrm {rad}}}
\newcommand{\abel}{{\mathrm {abel}}}
\newcommand{\codim}{{\mathrm {codim}}}
\newcommand{\ind}{{\mathrm {ind}}}
\newcommand{\Res}{{\mathrm {Res}}}
\newcommand{\Ann}{{\mathrm {Ann}}}
\newcommand{\Ext}{{\mathrm {Ext}}}
\newcommand{\Alt}{{\mathrm {Alt}}}
\newcommand{\AAA}{{\sf A}}
\newcommand{\SSS}{{\sf S}}
\newcommand{\CC}{{\mathbb C}}
\newcommand{\CB}{{\mathbf C}}
\newcommand{\RR}{{\mathbb R}}
\newcommand{\QQ}{{\mathbb Q}}
\newcommand{\ZZ}{{\mathbb Z}}
\newcommand{\KK}{{\mathbb K}}
\newcommand{\NB}{{\mathbf N}}
\newcommand{\ZB}{{\mathbf Z}}
\newcommand{\OB}{{\mathbf O}}
\newcommand{\EE}{{\mathbb E}}
\newcommand{\PP}{{\mathbb P}}
\newcommand{\GC}{{\mathcal G}}
\newcommand{\HC}{{\mathcal H}}
\newcommand{\AC}{{\mathcal A}}
\newcommand{\BC}{{\mathcal B}}
\newcommand{\GA}{{\mathfrak G}}
\newcommand{\SC}{{\mathcal S}}
\newcommand{\TC}{{\mathcal T}}
\newcommand{\DC}{{\mathcal D}}
\newcommand{\LC}{{\mathcal L}}
\newcommand{\RC}{{\mathcal R}}
\newcommand{\CL}{{\mathcal C}}
\newcommand{\EC}{{\mathcal E}}
\newcommand{\GCD}{\GC^{*}}
\newcommand{\TCD}{\TC^{*}}
\newcommand{\FD}{F^{*}}
\newcommand{\GD}{G^{*}}
\newcommand{\HD}{H^{*}}
\newcommand{\hG}{\hat{G}}
\newcommand{\hP}{\hat{P}}
\newcommand{\hQ}{\hat{Q}}
\newcommand{\hR}{\hat{R}}
\newcommand{\GCF}{\GC^{F}}
\newcommand{\TCF}{\TC^{F}}
\newcommand{\PCF}{\PC^{F}}
\newcommand{\GCDF}{(\GC^{*})^{F^{*}}}
\newcommand{\RGTT}{R^{\GC}_{\TC}(\theta)}
\newcommand{\RGTA}{R^{\GC}_{\TC}(1)}
\newcommand{\Om}{\Omega}
\newcommand{\eps}{\epsilon}
\newcommand{\varep}{\varepsilon}
\newcommand{\al}{\alpha}
\newcommand{\chis}{\chi_{s}}
\newcommand{\sigmad}{\sigma^{*}}
\newcommand{\PA}{\boldsymbol{\alpha}}
\newcommand{\gam}{\gamma}
\newcommand{\lam}{\lambda}
\newcommand{\la}{\langle}
\newcommand{\ra}{\rangle}
\newcommand{\hs}{\hat{s}}
\newcommand{\htt}{\hat{t}}
\newcommand{\sgn}{\mathsf{sgn}}
\newcommand{\SR}{^*R}
\newcommand{\tn}{\hspace{0.5mm}^{t}\hspace*{-0.2mm}}
\newcommand{\ta}{\hspace{0.5mm}^{2}\hspace*{-0.2mm}}
\newcommand{\tb}{\hspace{0.5mm}^{3}\hspace*{-0.2mm}}
\def\skipa{\vspace{-1.5mm} & \vspace{-1.5mm} & \vspace{-1.5mm}\\}
\newcommand{\tw}[1]{{}^#1\!}
\renewcommand{\mod}{\bmod \,}
\newcommand{\edit}[1]{{\color{red} #1}}

\newcommand{\carolinacomment}{\textcolor{cyan}}

\marginparsep-0.5cm

\renewcommand{\thefootnote}{\fnsymbol{footnote}}
\footnotesep6.5pt

\title{The Field of Values of the Height Zero Characters}
\author[G. Navarro]{Gabriel Navarro}
\address{Departament de Matem\`atiques, Universitat de Val\`encia, 46100 Burjassot,
Val\`encia, Spain}
\email{gabriel@uv.es}

\author[L. Ruhstorfer]{Lucas Ruhstorfer}
\address{School of Mathematics and Natural Sciences University of Wuppertal, 
	42119 Wuppertal, Germany}
\email{ruhstorfer@uni-wuppertal.de}

\author[P. H. Tiep]{Pham Huu Tiep}
\address{Department of Mathematics, Rutgers University, Piscataway, NJ 08854, USA}
\email{tiep@math.rutgers.edu}

\author[C. Vallejo]{Carolina Vallejo}
\address{Dipartimento di Matematica e Informatica `Ulisse Dini', 50134 Firenze (Italy)}
\email{carolina.vallejorodriguez@unifi.it}

\thanks{The research of the first author is supported by Ministerio de Ciencia
e Innovaci\'on
PID2019-103854GB-I00. The third author gratefully acknowledges the support of the NSF (grants DMS-1840702 and DMS-2200850),
the Simons Foundation, and the Joshua Barlaz Chair in Mathematics. The fourth author acknowledges support from the Rita Levi Montalcini Program (bando 2019)
and from the INdAM-GNSAGA}
\thanks{Part of this work was done when the third author visited Princeton University and MIT. 
It is a pleasure to thank Princeton University and MIT for generous hospitality and stimulating 
environment.}



\begin{abstract}
We determine what are the fields of values of the irreducible
$p$-height zero characters of all finite groups for $p=2$; we
conjecture what they should be for odd primes, and reduce this statement
to a problem on blocks of quasi-simple groups.
\end{abstract}

\maketitle

\section{Introduction} 
%As a consequence of the Kronecker-Weber theorem it is not difficult to prove that 
Every Abelian 
number field can be realized as the field of values of a complex
 irreducible character of a finite group (see, for instance, Theorem 2.2 of \cite{NT}).
 Motivated by the McKay conjecture and the McKay-Navarro conjecture
\cite[Conjecture A]{N1}, it is of great interest to characterize the field of values of the irreducible characters of degree 
not divisible by a fixed prime $p$. This was established in \cite{NT} for $p=2$ following seminal work on the
fields of values of odd-degree characters in \cite{ILNT}.
 For instance, the quadratic field of values of the odd-degree irreducible characters of all finite groups 
 are exactly the fields $\Q$, $\Q(i)$ and 
 $\Q(\sqrt d)$, where  $d \equiv 1$ mod 4 is a square-free number.
 
 The characters of degree not divisible by $p$ constitute only a part of the Brauer's 
 $p$-height zero characters, namely those %which have $p$-height zero 
 that lie in $p$-blocks of maximal defect. 
 The McKay conjecture admits a version for Brauer's $p$-blocks 
  (the celebrated Alperin--McKay conjecture)  where characters of degree not divisible by $p$ are replaced by $p$-height zero characters, and this conjecture
also admits a version taking into account the action of Galois automorphisms \cite[Conjecture B]{N1} (which is sometimes referred to as the Alperin-McKay-Navarro conjecture).

 The main question now is: What are then the field of values of the 2-height zero characters? 
 As we have mentioned,
  $\Q(\sqrt 3)$, say,  cannot be the field of values of an irreducible character of odd degree,
 but it is easy to find many 2-height zero characters having this field of values.
 (For instance, in a double cover of $\sf S_5$.)
 However, $\Q(\sqrt 2)$ or $\Q(\sqrt{-2})$, say, do  not appear to be the field of values of
 any 2-height zero character.

%As happened in \cite{ILNT},
%the main case is to study the field of values of the height zero characters which are not 2-rational. (Recall that if $\chi \in \irr G$ and $\Q(\chi)$ is the smallest field containing the values of $\chi$, then $\chi$ is $p$-rational if it is contained in a cyclotomic field $\Q_m=\Q(\zeta_m)$ for some $m$ not divisible by $p$.)
% 
 
 When studying fields of values of characters, character conductors are a fundamental invariant.
 If $\chi$ is a character of a group $G$ and $\Q(\chi)$ is the smallest field extension of $\Q$ containing the values of $\chi$,
 then we define $c(\chi)$, the conductor of $\chi$, to be the smallest integer $n$ such that 
 $\Q(\chi)$ is contained in the $n$-th cyclotomic field $\Q_n=\Q(e^{2\pi i / n})$.
 If $F$ is any subfield of $\C$, then we write $F(\chi)=\langle F, \Q(\chi) \rangle$.
 
 \medskip
 
 The following is the main result of this paper. Its proof uses the Classification of Finite Simple Groups, together
 with the work of \cite{BDR}, and its refinement by \cite{KL}.
 
\begin{thmA}
Let $G$ be a finite group, and 
let $\chi \in \irr G$ of 2-height zero. Write $c(\chi)=2^a m$, where $m$ is odd and $a\ge 0$.
Then $\Q_{2^a} \sbs \Q_m(\chi)$.
\end{thmA}

In fact, the fields of values of the 2-height zero irreducible characters can be characterized in the following way.
Let $\mathcal F_2$ be the set of Abelian number fields $F$ such that $\Q_n=\langle \Q_m, F \rangle$,
where here $n$ is the conductor of the field $F$, and $n=2^am$, for some odd number $m$.

\begin{thmB}
The set consisting of the fields of values of the 2-height zero characters of finite groups is exactly  $\mathcal F_2$.
\end{thmB}

As a consequence of Theorem B we obtain that the following are the quadratic number fields that appear as fields of values of  $2$-height zero
characters.

\begin{corC}
Let $F$ be a quadratic number field. Then $F=\Q(\chi)$ for some 2-height zero character $\chi$ if and only if 
 $\Q(\chi)=\Q(\sqrt d)$ where $d \neq 1$ is an odd square-free integer.
\end{corC}
    
In this context, it is natural to wonder what happens for odd primes. The fields of values of the characters of degree not divisible by $p$ are conjectured to be precisely the Abelian number fields
$F$ such that $[\Q_{p^a}:\Q_{ p^a} \cap F]$ is not divisible by $p$ in \cite[Conjecture C]{NT}.
This conjecture does not seem to follow from the McKay-Navarro conjecture \cite[Conjecture A]{N1}.
 As happens in the case of characters of degree not divisible by $p$,  we can only conjecture what the fields of values of the $p$-height zero characters should be for odd primes. This is Conjecture D below. 
 The novelty is that
 we can show that the statement of Conjecture D follows from the statement of the Alperin-McKay-Navarro conjecture.

 For any prime $p$, let $\mathcal F_p$ be the set of Abelian number fields $F$ with conductor
 $n=p^a m$, where $p$ does not divide $m$, such that the degree  $|\Q_n: \langle \Q_m, F\rangle |$ is not divisible by $p$.
 Notice that  the fields
$F$ with $p$-part of the conductor $p^a$ such that $[\Q_{p^a}:\Q_{ p^a} \cap F]$ is not divisible by $p$ are a subclass contained in $\mathcal F_p$.

\begin{conjD}
The set of fields of values of the $p$-height zero characters of finite groups is exactly  $\mathcal F_p$.
\end{conjD}

%If $G$ is the group $C_{63}:C_3$, where the $C_3$ raises the $C_7$ to the square, and $C_9$ to the 25,
%this is SmallGroup(189,4), then there is a character of degree 3, of height zero, such that the field of values
%is $F=NF(63,[1,25,58])$, but $F$ cannot be realized by a character of $3'$-degree, since $\Q_9: F \cap \Q_9$
%has degree divisible by 3. The same in $PSL(6,2)$, which has a character of height zero
%and field of values $F=NF(63,[1,2,4,16,32])$.

We show that any field in $\mathcal F_p$ is the field of values of a $p$-height zero character in Theorem \ref{G}.
Hence settling one of the containments in Conjecture D (and also reducing the proof of Theorem B to proving Theorem A).
We reduce the verification of  the other containment to a problem on quasi-simple groups in Theorem \ref{reduction}.

We show that the statement of Conjecture D follows from the statement of the Alperin-McKay-Navarro conjecture \cite[Conjecture B]{N1} in Theorem \ref{cons}. 
%(We remark once again that the conjecture characterizing the fields of values of $p'$-degree characters does not follow from the McKay-Navarro conjecture \cite[Conjecture A]{N1}.)
 A fundamental part of our work is devoted to show that Theorem A is true for quasi-simple groups. 
 We believe that this will be useful in the final verification of the Alperin-McKay-Navarro conjecture.

\section{Conductors}

Let us start by recording some elementary results on characters and conductors that we will frequently use.
Recall that if $\psi$ is a character of a finite group, then $\psi(g)$ is a sum of $o(g)$-th roots of unity for $g \in G$, and therefore
the field of values $\Q(\psi)$, which is the smallest field containing $\psi(g)$ for all $g \in G$,  is contained in $\Q_{|G|}=\Q(e^{2\pi i/|G|})$.
The conductor $c(\psi)$ is the smallest $n$ such that $\Q(\psi) \sbs \Q_n=\Q(e^{2 \pi i /n})$. Therefore
$c(\psi)$ divides $|G|$. Moreover, $\Q(\psi) \sbs \Q_m$ if and only if $c(\psi)$ divides $m$. If $F$ is an Abelian number field, that is $F\sbs \C$ and $F/\Q$ is a Galois extension with ${\rm Gal}(F/\Q)$ abelian,
then the Kronecker-Weber theorem implies that $F \sbs \Q_n$ for some $n$ and $c(F)$, the conductor of $F$, is the smallest such $n$.  By elementary Galois theory, recall that
$c(\langle F_1, F_2\rangle)$ is the least common multiple of $c(F_1)$ and $c(F_2)$. 

\medskip

In this paper, if $p$ is a prime and $n \ge 1$ is an integer, then
$n_p$ is the largest power of $p$ dividing $n$, and $n_{p'}=n/n_p$. We call $n_p$ the $p$-part of $n$ and $n_{p'}$ the $p'$-part of $n$.
 For a fixed prime $p$,   we are interested in the $p$-parts of conductors.
If $\psi$ is a character and $c(\psi)_p=1$, then $\psi$ is called $p$-rational.
If $p=2$,  $\psi$ is either 2-rational
or $c(\psi)_2\ge 4$. 
Notice that if $\chi$ is linear
character then $c(\chi)=o(\lambda)$ unless $o(\lambda)_2=2$
in which case $c(\chi)=o(\lambda)/2$.

 \begin{lem}\label{useful}
    Let $p$ be a prime.
    Suppose that $\chi \in \irr G$, and write $c(\chi)=p^a m$, where $n$ is not divisible by $p$.
    If  $n$ is a natural number not divisible by $p$
   with $\Q_{p^f} \sbs \Q_{pn}(\chi)$, then $\Q_{p^f} \sbs \Q_{pm}(\chi)$.
   Moreover $f\le a$ unless possibly when $f=1$ and $a=0$. 
  % If $p$ is odd and $a\ge 1$, or $p=2$ and $f\ge 2$, we have that $f\le a$.
    \end{lem}
    
     \begin{proof}
    By replacing $n$ by $mn$, we we may assume that $m$ divides $n$. 
    If $a=0$ and $f=1$ then $\Q_p\sbs \Q_{pm}(\chi)$.
    
    Hence we may assume that $a\ge 1$. Then $a \ge 2$ if $p=2$. In either case
    $\Q_{p^f}\sbs \Q_{pn}(\chi)\sbs \Q_{p^an}(\chi)=\Q_{p^a n}$ because $m$ divides $n$, so $f \leq a$.
    If $p=2$, we may also assume that 
    $f\ge 2$, because otherwise the result is trivial.
    
    Write $F=\Q_n$, $K=\Q_m$, $L=\Q_{p^a m}$ and $E=\langle F, L \rangle=\Q_{p^a n}$.
    We have that $F \cap L=K$. Let $J=\Q_{pm}(\chi)$, so that $K \sbs J \sbs L$.  Let  $M=\Q_{pn}(\chi)=\langle F, J\rangle$.
 Since $\Q_{p^f}\sbs  M \sbs \Q_{p^a n}$,  
  we have that $f \le a$. 
 Now, $\Q_{p^f} \sbs M \cap L=J$, by Lemma 2.6(i) of \cite{NT}, for instance.  
    \end{proof}

%\begin{lem}\label{silly}
%Let $p$ be a prime. Suppose that $\chi$ and $\psi$ are characters of groups $G$ and $H$.
%If $\Q_m(\chi)=\Q_m(\psi)$ for some   $m$ not divisible by $p$, then $c(\chi)_p=c(\psi)_p$.
%\end{lem}
%\begin{proof}
%
%
%    Suppose that $c(\chi)=p^as$, where $s$ is not divisible by $p$,
%     and $c(\psi)=p^bt$, where $t$ is not divisible by $p$, and $a,b\ge 0$.
%    Then $\Q(\psi) \sbs \Q_m(\psi)=\Q_m(\chi) \sbs \Q_{p^a sm}$. Therefore $c(\psi)$ divides $p^asm$,
%    and thus $p^b\le p^a$.  By reversing $\psi$ and $\chi$, we have that  $p^a \le p^b$.
%     \end{proof}

\begin{lem}\label{triv}
Let $p$ be a prime. Suppose that $\chi$ and $\psi$ are characters of groups $G$ and $H$.
Suppose that $\Q_{pn}(\chi)=\Q_{pn}(\psi)$ for some   $n$ not divisible by $p$.
\begin{enumerate}
\item
If $n$ divides $m$, then $\Q_{pm}(\chi)=\Q_{pm}(\psi)$.

\item
If $p=2$, or $p$ is odd and  $c(\chi)_p,c(\psi)_p \ge p$, then $c(\chi)_p=c(\psi)_p$.
\end{enumerate}
\end{lem}
    
    \begin{proof}
  To prove part (i) just notice that 
    $$\Q_{pm}(\chi)=\Q(\zeta_p,\zeta_m, \chi)=\Q(\zeta_p,\zeta_n, \zeta_m, \chi)=\Q_{np}(\chi)(\zeta_m)=\Q_{np}(\psi)(\zeta_m)=\Q_{pm}(\psi) .$$

    To prove part (ii) notice that $ \Q(\psi)\sbs \Q_{pn}(\chi)\sbs \Q_{c(\chi)_pm}$ with $m=nc(\chi)_{p'}$.
    In particular $c(\psi)_p$ divides $c(\chi)_p$. By reversing the roles played by $\chi$ and $\psi$ we obtain the result.
%    Suppose that $c(\chi)=p^as$, where $s$ is not divisible by $p$,
%     and $c(\psi)=p^bt$, where $t$ is not divisible by $p$, and $a,b\ge 0$.
%    Then $\Q(\psi) \sbs \Q_{pm}(\psi)=\Q_{pm}(\chi) \sbs \Q_{p^a sm}$, if $a\ge 1$ or $p=2$. 
%    Therefore $c(\psi)=p^bt$ divides $p^asm$, and we conclude that $b\le a$. By reversing $\chi$ and $\psi$,
%    we obtain the result. 
    \end{proof}

    \section{Fields and Height Zero Characters}
           
       Our notation for blocks follows \cite{N2}.
We will frequently use the following facts on height zero characters. 

\begin{thm}\label{ht0}
Let $B$ be a $p$-block of a finite group $G$, and let $\chi \in \irr B$ with height zero.
\begin{enumerate}
\item If $\psi^G=\chi$, where $\psi \in \irr H$ of some subgroup $H$ of $G$, then
$\psi$ has height zero in its $p$-block, and any defect group of the block of $\psi$
is a defect group of $B$.

\item
If $N \nor G$ and $\theta \in \irr N$ is under $\chi$, then $\theta$ has height zero.
\end{enumerate}
\end{thm}    

\begin{proof}
Part (i) is Proposition 2.5(e) of \cite{NS}. Part (ii) is due to M. Murai, and is 
Proposition 2.5(a) of \cite{NS}.
\end{proof}
    
    \begin{thm}\label{cliff}
    Let $p$ be a prime, and suppose that $\chi \in \irr G$  has height zero in its $p$-block.
    Let $N \nor G$, let $\theta \in \irr N$ be under $\chi$, and let $\psi \in \irr{T|\theta}$
    be the Clifford correspondent of $\chi$ over $\theta$.
    Then $\psi$ and $\theta$ have height zero. Also,  $\Q_{pn}(\chi)=\Q_{pn}(\psi)$, where $n=|G|_{p'}$. Therefore, if $p=2$ or $p$ odd and $c(\chi)_p \ge p$, then $c(\chi)_p=c(\psi)_p$.    \end{thm}
    
    \begin{proof}
    We have that $\psi$ and $\theta$ have height zero
    by Theorem \ref{ht0}. We argue by induction on $|G:N|$. Since $\psi^G=\chi$, we have that $\Q(\chi) \sbs \Q(\psi)$. 
    Let $T^*$ be the semi-inertia group of $\theta$ in $G$ consisting of the 
    elements $g \in G$ for which there is some $\sigma \in {\rm Gal}(\Q(\theta)/\Q)$ such that
    $\theta^g=\theta^\sigma$, as in Problem 3.9 of  \cite{N3}. Let $\eta=\psi^{T^*}$.
    Since $\eta^G=\chi$, then we have that $\eta$ has height zero.
    Also, $\Q(\chi)=\Q(\eta)$, $T\nor T^*$ and $T^*/T$ is abelian.
   
    If $T^*<G$, by induction, we have that $\Q_{p|T^*|_{p'}}(\eta)=\Q_{p|T^*|_{p'}}(\psi)$.
    By Lemma \ref{triv}(i), we have that $\Q_{pn}(\psi)=\Q_{pn}(\eta)=\Q_{pn}(\chi)$,
    and we are done.  
   Thus we may assume that $T^*=G$. Then $T \nor G$, and that $G/T$ is abelian.  By induction,
    we may assume that $T=N$, and $\psi=\theta$. If $M$ is a maximal normal subgroup of $G$ with $N<M<G$, then
    again by induction (and using Lemma \ref{triv}(i)), $\Q_{pn}(\theta^M)=\Q_{pn}(\theta)$. Hence it is enough to prove
    the statement in the case where $G/N$ has prime order.
    
   Now $G/N$ has prime order.
    Let $\sigma \in {\rm Gal}(\Q_{pn}(\theta)/\Q_{pn}(\chi))$.  We want to show $\sigma$ is trivial. Assume that $\sigma\neq 1$. Notice that $\sigma$ is a $p$-element,
    since ${\rm Gal}(\Q_{|G|}/\Q_{pn})\leq {\rm Gal}(\Q_{|G|}/\Q_{p n})\cong {\rm Gal}(\Q_{|G|_p}/\Q_p)$ is a $p$-group. 
    By Clifford's theorem, we have that $\theta^\sigma=\theta^g$ for some
    $g \in G$.  Also $\theta^\sigma\neq \theta$ because $\sigma$ is not trivial.
    In particular $\langle gN \rangle=G/N$ is a group of  order $p$.    
   Let $b$ be the block of $\theta$. Since $\sigma$ fixes $p^\prime$-roots of unity,
    it follows that $b^\sigma=b$. (Use, for instance, Theorem 3.19 of \cite{N2}.) Then $b^g=b$ and $b$ is $G$-invariant. Then we apply Corollary 9.6 and Corollary 9.18 of \cite{N2},
    and conclude that $\theta$ is $G$-invariant, a contradiction. 
    
    The second part of the statement follows from Lemma \ref{triv}. Notice if $p$ is odd then $c(\chi)_p\ge p$ implies that $c(\psi)_p\ge 1$ because if $c(\psi)_p=1$, then $\psi$ is $p$-rational, and $\chi=\psi^G$ is also
    $p$-rational. 
    \end{proof}
    
    Notice that the hypothesis on the odd case of the second statement of the above theorem is necessary: if
    $p=3$,  $\chi \in \irr{{\sf S}_3}$ has degree 2 and $\psi \in \irr{N}$ is under $\chi$ with $|N|=3$, then $c(\chi)_3=1$ but $c(\psi)_3=3$. 
    \medskip
    \begin{cor}\label{cor}
    Let $N\nor G$, let $\chi \in \irr G$ of  height zero in its $p$-block, and let $\theta$ be an irreducible constituent of $\chi_N$.
    If $p$ is odd, assume that $c(\chi)_p\ge p$. Then $c(\theta)_p \le c(\chi)_p$.
    In particular, if $p=2$ and $\chi$ is 2-rational, then $\theta$ is $2$-rational. 
    \end{cor}
    
    \begin{proof}
    By Theorem \ref{cliff}, we may assume that $\theta$ is $G$-invariant. Then $\chi_N=e\theta$, $\Q(\theta) \sbs \Q(\chi)$, and
    the statement is clear.
    \end{proof}
     \medskip
     
     Of course, Corollary \ref{cor} is about height zero characters. (Consider, for instance, $G={\sf D}_8$, $\chi \in \irr G$ of  degree 2,
     and $N$ a cyclic subgroup of $G$ of order 4.)
     
     \medskip
     
   Next we prove the normal defect case of Theorem A. We first need a lemma. Suppose that  $\chi \in \irr G$ lies in a block $B$ with defect group $D \nor G$. 
    Let $C=\cent GD$. Let $b$ a block of $CD$ covered by $B$. By Corollary 9.21 of \cite{N2},
    we have that $B=b^G$ is the only block of $G$ covering $b$. By Theorem 9.26 of \cite{N2},
    we have that $b$ has defect group $D$.  By Theorem 9.12 of \cite{N2}, there is a unique irreducible
    character $\theta \in \irr{b}$ such that $D \sbs \ker\theta$. This character has defect zero, viewed as a character 
    of $CD/D$, and it is called the canonical character of $b$, which is
    uniquely defined up to $G$-conjugacy. The irreducible characters of $b$ are described in Theorem 9.12 of \cite{N2}.

    \begin{lem}\label{qp}
    Suppose that $\chi \in \irr G$ has height zero and belongs to a $p$-block $B$
    with a normal defect group $D$. Let $C=\cent GD$ and $Z=\zent D$.  If $\chi_{CD}$
    %and $\chi_D$ are homogeneous,
    is homogeneous, then $\chi_D$ is homogeneous, the canonical character $\theta \in \irr{CD/D}$
    of $B$ is $G$-invariant, and $G/CD$ is a $p'$-group.
    If %$\eta$ is the irreducible constituent of $\chi_{CD}$ and 
    $\lambda$ the irreducible
    constituent of $\chi_D$, then $\lambda$ is linear and $c(\chi)_p=c(\lambda)$. %if $p$ is odd, or $p=2$ and $o(\lambda)\ge 4$.
    %Also, if $p=2$, then $\chi$ is 2-rational if and only if $o(\lambda)\le 2$.        
    \end{lem}
    
    \begin{proof}
     Let $\eta \in \irr{CD}$ be under $\chi$. By  Theorem \ref{ht0}, we have that $\eta$ has height zero.
    By Theorem 9.12 of \cite{N2}, we know that we can write $\eta=\theta_\lambda$,
       where $\lambda \in \irr D$ is linear
       and $\theta \in \irr{CD/D}$ has defect zero. Moreover, $\eta(x)=0$ if $x_p \not\in D$, and $\eta(x)=\theta(x_{p'})\lambda(x_p)$ if $x_p \in D$.  Thus $\eta_D=\theta(1)\lambda$. Since  $\chi_{CD}$
       is homogeneous, it follows that $\chi_D$ is homogeneous then $\lambda$ and $\eta$ are $G$-invariant.
        We claim that $\theta\in\irr{CD/D}$ is $G$-invariant. View $\theta$ as a character of $C/Z$. 
       Let $x \in C$ and $g \in G$. Since $\theta \in \irr{C/Z}$ has $p$-defect zero, if $xZ$ is $p$-singular, then $\theta(x)=0=\theta(x^g)$ because $x^gZ$ is also $p$-singular.
       If $xZ$ is $p$-regular, then $x^gZ$ is also $p$-regular, $xZ=x_{p'}Z$ and $x^gZ=(x_{p'})^gZ$. Since $\eta$ is $G$-invariant, we have that
        $\theta(x_{p'})\lambda(x_p)=\eta(x)=\eta(x^g)=\theta((x_{p'})^g)\lambda((x_{p'})^g)$. Since $\lambda$ is $G$-invariant
       and linear, we deduce that $\theta(x_{p'})=\theta((x_{p'})^g)$. Now $\theta(x)=\theta(x_{p'})=\theta((x_{p'})^g)=\theta(x^g)$ because $Z$ is contained in the kernel of $\theta$, and we deduce that $\theta$ is $G$-invariant. Therefore, we have that $G/CD$ has order coprime to $p$ by Theorem 9.22 of \cite{N2}.
       
      Since $\chi_{CD}=e\eta$ and $\eta_D=\theta(1)\lambda$, we have that $\Q_{c(\lambda)} \sbs \Q(\eta)$. Now, $\theta$ is $p$-rational, because it is a defect zero character, and therefore,
      $\Q(\theta) \sbs \Q_m$, where $m$ is a $p'$-number.  Then, using the formula for the values of $\eta$, we
      have that  $\Q_{c(\lambda)} \sbs \Q(\eta) \sbs \Q_{c(\lambda) m}$. Therefore, $c(\lambda)$ divides $c(\eta)$ which divides $c(\lambda)m$, implying
      that $c(\eta)_p=c(\lambda)$.
     We apply Lemma 4.2(ii) of \cite{NT}, and we get that $c(\chi)_2=c(\eta)_2$ if $p=2$ and
       $c(\chi)_p=c(\eta)_p$ if $p$ is odd and $c(\lambda)=c(\eta)_p>1$.
      If $p$ is odd and $c(\lambda)=c(\eta)_p=1$, then $\lambda=1_D$.
      Therefore $\eta=\theta$, and $\chi$ has $p$-defect zero, so $\chi$ is $p$-rational and $c(\chi)_p=1$.
      In any case, we conclude that $c(\chi)_p=c(\lambda)$. 
           \end{proof}

Next we prove Theorem A, and one of the containments of Conjecture D,  in the case of blocks with a normal defect group.

            \medskip
%Recall that if $\theta$ is a $p$-defect zero of a finite group $G$, then  $\Q(\theta) \sbs \Q_{|G|_{p'}}$.
%\medskip

    \begin{lem}\label{normald}
    Let $\chi \in \irr B$ of height zero,
 where $B$ is a $p$-block with  a normal defect group $D$.
 Write $c(\chi)=p^am$, where $m$ is not divisible by $p$.
 Then
 $\Q_{p^a} \sbs \Q_{pn}(\chi)$.
  \end{lem}

\begin{proof} First notice that we may assume that $a\ge 2$ as otherwise the result trivially holds. 
We argue by induction on $|G|$. 

% Let $\lambda \in \irr D$ be an irreducible constituent of $\chi_D$.
%By Theorem \ref{cliff}, we have that $\lambda$ has height zero, and therefore $\lambda$
%is linear. 
Write $C=\cent GD$.
Let $\eta \in \irr{CD}$ be an irreducible constituent of $\chi_{CD}$.
Let $T=G_\eta$ be the stabilizer of $\eta$ in $G$ and $\psi \in \irr{T|\eta}$ be
the Clifford correspondent of $\chi$ over $\eta$.
By Theorem \ref{ht0}, we know that $\psi$ has height zero and that $D$
is a defect group of its block. By Theorem \ref{cliff}, 
we have that $\Q_{pn}(\chi)=\Q_{pn}(\psi)$ where $n=|G|_{p'}$
 and $c(\chi)_p=c(\psi)_p$.
Assume that $T<G$. By induction $\Q_{p^a}\sbs \Q_{pc(\psi)_{p'}}(\psi)\sbs \Q_{pn}(\psi)=\Q_{pn}(\chi)$.
By Lemma \ref{useful} we conclude that $\Q_{p^a}\sbs \Q_{pm}(\chi)$.

Hence $T=G$ and we are under the hypotheses of Lemma \ref{qp}. 
If $\lambda \in \irr D$ lies under $\chi$, then $p^a=c(\chi)_p=c(\lambda)$. Notice that
$\chi_D=f\lambda$ and hence $\Q_{p^a}\sbs \Q(\chi)\sbs \Q_{pm}(\chi)$.
\end{proof}

The next results are key to understand the statement of Conjecture D when the group possesses  a normal subgroup of index $p$
(a fundamental step in our reduction theorem).

\begin{lem}\label{elquefaltaba}
Suppose that $G/N$ is a $p$-group, and $b$ is a $G$-invariant $p$-block of $N$ covered 
by a block $B$ of $G$ with defect group $D$. Suppose that $D_0=D\cap N \nor G$. 
Then $G=DN$ and $b$ has a $G$-invariant height zero $p$-rational irreducible character.
\end{lem}

\begin{proof}
By Corollary 9.6 of \cite{N2}, $B$ is the unique $p$-block covering $b$. Then Theorem 9.17 of \cite{N2} implies that $G=ND$ and $D_0$ is the
unique defect group of $b$. 
Let $C=\cent N{D_0}$. Notice that $CD_0\nor G$.
By the Fong-Reynolds correspondence, Theorem 9.14 of \cite{N2}, we can find
   $e$ 
a block of   $CD_0$ covered by $b$ such that the block $b_T$ of 
 $T=G_e$, the stabilizer of $e$ in $G$, inducing $B$ and covering $e$ has defect group $D$. 
 Notice that $e$ has defect group $D_0$.
 Since $b$ is $G$-invariant, notice that $TN=G$, $b_T$ covers a unique block $f$ of $N_e=N\cap T$ that induces $b$ and covers $e$. 
 By induction and the Fong-Reynolds correspondence, we may assume that $e$ is $G$-invariant. 
 Then we have  that  $N/CD_0$ is a $p'$-group, by Theorem 9.22 of \cite{N2}. Since $e$ is $G$-invariant,
 we have that the canonical character $\theta \in \irr{CD_0/D_0}$ of $e$ is $G$-invariant.
 By Theorem 13.31 of \cite{I}, some irreducible constituent $\xi$ of $\theta^N$ is $D$-invariant.
 Thus $\xi$ is $G$-invariant. Since $b$ is the only block of $N$ that covers $e$, we have that $\xi \in \irr b$. Also,
 $\xi$ is $p$-rational, because it has defect zero considered as a character of $N/D_0$. 
 It also has height zero because $\theta$ has height zero and $N/CD_0$ is a $p'$-group.
\end{proof}

\begin{lem}\label{elem}
Suppose that $G/N$ is a cyclic $p$-group, and $\theta \in \irr N$ is $G$-invariant of $p$-height zero.
Then every $\chi \in \irr{G|\theta}$ has $p$-height zero. Also, if $D$ is a defect group of the block
of $\chi$, then $DN=G$ and $D\cap N$ is a defect group of the block of $\theta$.
\end{lem}

\begin{proof}
Let $b$ be the block of $\theta$. Let $B$ be the unique $p$-block of
$G$ covering $b$ by Corollary 9.6 of \cite{N2}.  Let  $\chi \in \irr{G|\theta}$, so that $\chi \in \irr B$.  Since $b$ is $G$-invariant, we
    have that $G=DN$, where $D$ is a defect group of $B$, and $D_0=D\cap N$ is a defect group of $b$ by Theorem 9.17 of \cite{N2}.
 We have that $\chi_N=\theta$ because $G/N$ is cyclic and $\theta$ is $G$-invariant 
     (using Theorem 5.1 of \cite{N3} and the Gallagher correspondence Corollary 1.23 of \cite{N3}). Then $\chi(1)_p=\theta(1)_p=|N:D_0|=|G:D|$. Thus $\chi$ has height zero.
\end{proof}

     \begin{lem}\label{index2lem}
    Suppose that $G/N$ is a $p$-group. Let $\theta \in \irr N$ be of $p$-height zero   and $G$-invariant.
    Let $n=|G|_{p'}$.
    Let $D_0$ be a defect group of the block of $\theta$, let $H=\norm G{D_0}$.
    Then there exists an $H$-invariant  $\varphi \in \irr{N\cap H}$ of $p$-height zero such that
    $[\theta_{H \cap N}, \varphi]\not\equiv 0 \mod p$, and $\Q_{pn}(\varphi) \sbs \Q_{pn}(\theta)$. 
        \end{lem}

    \begin{proof}   
     Let $b$ be the $p$-block of $\theta$ and let $B$ be the only $p$-block of $G$ covering $b$.
     Since $b$ is $G$-invariant, we
    have that $G=DN$, where $D$ is a defect group of $B$, and $D_0=D\cap N$ is a defect group of $B_0$,
     by Theorem 9.17 of \cite{N2}.

   Since $D \sbs H$, note that $G=HN$. Let  $M=H\cap N=\norm N{D_0}$. Then $H=MD$.  Let $e$ be the Brauer correspondent block of $M$ 
   (with defect group $D_0$) inducing $b$. By the Harris--Kn\"orr Theorem  \cite[Theorem 9.28]{N2},
   there is a unique block $E$ of $H$ covering $e$ that induces $B$. This block 
   $E$ has defect group $D$.

     Let $\mathcal U={\rm Gal}(\Q_{|G|}/\Q_{pn}(\theta))$. Notice that $\mathcal U\leq {\rm Gal}(\Q_{|G|}/\Q_{pn})$ which is a $p$-group.   
     We must then work to show that $e$ has a $D\times \mathcal U$-invariant height zero character $\varphi$ with $[\theta_M, \varphi]\not\equiv 0$ mod $p$.
     We will use the $\tilde{} $ construction as in \cite[page 27]{N2} and the fact that $e$ possesses an irreducible character $\psi$ which is $H$-invariant and $p$-rational
     by Lemma \ref{elquefaltaba}.

     Let $\delta=\theta_M$. Then 
     $\tilde\delta$ defined as $\tilde\delta(x)=|M|_p\delta(x)$, if $x \in M$ is $p$-regular, and 0, otherwise, is a generalized character by Lemma 2.15 of \cite{N2}.
     We can write
     $$\tilde \delta= \sum_{\xi \in \irr{M}}[\delta, \xi] \tilde \xi \, .$$
    By Lemma 6.5.(b) of \cite{N2}, we have that
     $${[\tilde \delta, \psi] \over \psi(1)} \not \equiv 0 \, {\rm mod}\, \mathcal P \, ,$$ 
     where
    $\mathcal P$ is the maximal ideal of ${\bf R}_M$ the localization of the ring of algebraic integers $\bf R$ at a prime $M$ containing $p$ (see \cite[page 16]{N2}). 
    Therefore
    $$\Lambda= \sum_{\xi \in \irr{M}}[\delta, \xi] {[\tilde \xi, \psi] \over \psi(1)} \not \equiv 0 \, {\rm mod}\, \mathcal P \, .$$
    Note that $[\tilde \xi, \psi]=[\xi, \tilde \psi]$ whenever $\xi \in \irr M$. 
    By Lemma 3.20 of \cite{N2}, recall that $[\xi, \tilde \psi]=0$ if $\xi$ is not in $e$, so that 
     $$\Lambda= \sum_{\xi \in \irr{e}}[\delta, \xi] {[\tilde \xi, \psi] \over \psi(1)} \not \equiv 0 \, {\rm mod}\, \mathcal P \, .$$
    By Lemma 3.22(a) of \cite{N2} ${[\tilde \xi, \psi]\over \psi(1)}\in {\bf R}_M \cap \Q$ for every $\xi \in \irr e$, so
    $\nu ({[\tilde \xi, \psi]\over \psi(1)})\ge 0$, where $\nu$ is the valuation function defined in \cite[page 64]{N2}.
    By Theorem 3.24 of \cite{N2},  we have that $\nu ({[\tilde \xi, \psi]\over \psi(1)})= 0$ if, and only if, $\xi$ has height zero in the $p$-block $e$ ($\xi \in {\rm Irr}_0(e)$).
    By Lemma 3.21 of \cite{N2} we have that ${[\tilde \xi, \psi]\over \psi(1)} \in \mathcal P$ whenever $\xi$ does not have height zero in $e$.
    Hence $$\Lambda \equiv \sum_{\xi \in {\rm Irr}_0(e)}[\delta, \xi] {[\tilde \xi, \psi] \over \psi(1)} \not \equiv 0 \, {\rm mod}\, \mathcal P\, .$$
    Consider $\Omega=\{ \xi \in {\rm Irr}_0(e) \  | \ [\delta, \xi]\not \equiv 0 \mod p \}$. We have that
    $$\Lambda\equiv \sum _{\xi \in \Omega} [\delta, \xi] {[\tilde \xi, \psi] \over \psi(1)} \not \equiv 0 \, {\rm mod} \, \mathcal P\, . $$
    The $p$-group $D \times \mathcal U$ acts on $\Omega$. Let $\Omega=\Omega_{\xi_1}\cup \ldots \cup \Omega_{\xi_r}$ be the orbit
    decomposition of $\Omega$. Then, given that $\delta$ and $\psi$ are $D \times \mathcal U$-invariant, we have that
    $$\Lambda\equiv \sum_{i=1}^r |\Omega_{\xi_i}| [\delta, \xi_i] {[\tilde {\xi_i}, \psi] \over \psi(1)} \not \equiv 0 \, {\rm mod} \, \mathcal P\, .  $$
    In particular there is some $D\times \mathcal U$-invariant $\varphi \in \Omega$. The $D$-invariance of $\varphi$ implies that $\varphi$ is $H$-invariant
    and the $\mathcal U$-invariance of $\varphi$ implies that  $\Q_{pn}(\varphi) \sbs \Q_{pn}(\theta)$.
   \end{proof}

The following is a consequence of an argument of J. Thompson, we refer the reader to Theorem 6.9 of \cite{N3}.

    \begin{lem}\label{mult}
    Suppose that $G/N$ is a $p$-group, and let $H \le G$ such that $G=NH$.
    Write $M=N\cap H$. Let $\theta \in \irr N$ be $G$-invariant, and let
    $\varphi \in \irr M$ be $H$-invariant such that $[\theta_M, \varphi] \not\equiv 0$ mod $p$.
    
    \begin{enumerate}[(a)]
    
    \item Suppose that $\xi \in \irr H$ extends $\varphi$. Then there is an extension $\chi \in \irr G$ of $\theta$
    such that $\Q(\chi) \sbs \Q(\xi,\theta)$.
    
    \item Suppose that $\chi \in \irr G$ extends $\theta$. 
    Then there is an extension $\xi \in \irr H$ of $\varphi$
    such that $\Q(\xi) \sbs \Q(\chi,\varphi)$.
    
    \item
    Suppose that   $G/N$ has order $p$. Then  $\Q_p(\chi,\varphi)=\Q_p(\theta,\xi)$
    for every $\chi \in \irr{G|\theta}$
    and $\xi \in \irr{H|\varphi}$.

    \end{enumerate}
    \end{lem}
    
    \begin{proof} Write $m=|G|$. In order to prove part (a), let $\sigma \in {\rm Gal}(\Q_m/\Q(\xi,\theta))$. 
    Since $\xi^\sigma=\xi$, in particular $\varphi^\sigma=\varphi$. 
    By  Theorem 6.9 of \cite{N3}, there is a unique extension $\chi \in \irr G$ of $\theta$ such that
    $\chi_H=\Psi \xi + \Delta$, where $\Delta$ is a character of $H$ or zero, with $[\Delta_M, \varphi]=0$
    and $\Psi$ is a character of $H/M$ with trivial determinant.
    Now, $\chi^\sigma$ is an extension of $\theta=\theta^\sigma$ and 
    $(\chi^\sigma)_H=\Psi^\sigma \xi + \Delta^\sigma$, where $[(\Delta^\sigma)_M, \varphi]=0$ and
    $\Psi^\sigma$ is a character of $H/M$ with trivial determinant. By the uniqueness of $\chi$, we get that $\chi=\chi^\sigma$.
    Part (b) is proved in the same way.
   
   We prove part (c). We have that $\theta$ extends to $G$ by Theorem 5.1 of \cite{N3}. In fact, every character in $\irr{G|\theta}$ is an extension of $\theta$ by the Gallagher correspondence. Let $\chi \in \irr{G|\theta}$.  By part (b), there is an extension
   $\xi \in \irr H$ of $\varphi$ such that $\Q(\xi) \sbs \Q(\chi,\varphi)$. By part (a), there is
   an extension $\chi^\prime \in \irr G$ of $\theta$ such that $\Q(\chi^\prime) \sbs \Q(\xi,\theta)$.
   Since $\chi^\prime=\lambda \chi$ for some $\lambda \in \irr{G/N}$, 
   we have that $\Q_p(\chi)=\Q_p(\chi^\prime)$. Since $\Q(\theta) \sbs \Q(\chi)$ and $\Q(\varphi)\sbs \Q(\xi)$, part (c) follows. 
    \end{proof}
   
   In order to treat later the case where $N$ is a normal subgroup of $G$ of index $p$, we need to extend Lemma \ref{normald}, and prove the statement of Conjecture D in a slightly more general case than
   the normal defect group case.
   
   \begin{lem}\label{extnormald}
   Suppose that $G/N$ has order $p$. Let $n=|G|_{p'}$.
   Let $\chi \in \irr G$ have $p$-height zero.
   Suppose that $\chi_N=\theta \in \irr N$ and the defect group $D_0$ of
   the block of $\theta$ is normal in $G$. 
  Then $\Q_{c(\chi)_p} \sbs \Q_{pn}(\chi)$.
   \end{lem}
   
   \begin{proof} Write $p^a=c(\chi)_p$. We may assume that $a\ge 2$ as the statement is trivially satisfied otherwise.
   We argue by induction on $|G|$.  
   
   Write $K=\cent N {D_0}D_0 \nor G$. Let $\eta$ be an irreducible constituent of $\chi_K$ and $\psi \in \irr{G_\eta | \eta}$ be the Clifford correspondence of $\chi$.
   Using Theorem \ref{cliff} and induction, we may assume that $\eta$ is $G$-invariant. In particular, $\theta_K=\chi_K$ is homogeneous.
   By Lemma \ref{qp},
   $\chi_{D_0}=\theta_{D_0}$ 
   is homogeneous. 
   %and $|N:K|$ is coprime to $p$.
   Write $\theta_{D_0}=\theta(1)\lambda$,
   where $\lambda \in \irr{D_0}$ is linear and $G$-invariant.
  %We have that $c(\chi)_p=o(\lambda)$ again using Lemma \ref{qp}.
  
  Let $D$ be a defect group of the block of $\chi$. We have that $G=ND$ and $D_0=N \cap D$, using for instance Lemma \ref{elem}. 
   Let $H=\norm G D$ and $C=\norm ND=N\cap H$. If $H=G$ then $\chi$ lies in a block of normal defect, and we are done by Lemma \ref{normald}.
   Hence we may assume that $H<G$.
   By Theorem A of \cite{NS}, there are a block $b'$ of $C$ with defect $D_0$
   and a $D$-invariant character $\theta'$ in $b'$ satisfying that
  $[\theta_C, \theta'] \equiv \pm 1 \, {\rm mod}\,  p$. Since $\lambda$ is $G$-invariant $\theta^\prime$ is the unique irreducible constituent of $\theta_C$
 such that $\theta^\prime(1)_p=|C:D_0|_p$, and   $[\theta_C, \theta^\prime]\not \equiv 0 \mod p$.
 The block $B'$ of $H=CD$ covering $\theta^\prime$ has defect group $D$,
 the block of $\theta^\prime$ has defect group $D_0$, and   $\theta^\prime$ has height zero.
 
 Given $\sigma\in {\rm Gal}(\Q_{|G|}/\Q_n)$, we have that $b^\sigma=b$ and $(b')^\sigma=b'$ because $\sigma$ fixes $p'$-roots of unity.
 Moreover $(\theta^\sigma)'=(\theta')^\sigma$ under the canonical correspondence given by Theorem A of \cite{NS} (because of part (c) in that statement, noticing that $\lambda^\sigma$ is also $G$-invariant).
 This implies
that $\Q_n(\theta^\prime)=\Q_n(\theta)$  by elementary Galois theory. 

Let $\xi \in \irr{H|\theta^\prime}$. Then $\xi$ extends $\theta^\prime$ and, by Lemma \ref{elem}, notice that $\xi$ has height zero.
  By Lemma \ref{mult}(c), we have that $\Q_p(\chi,\theta^\prime)=\Q_p(\theta, \xi)$.
  Since $\Q_n(\theta)=\Q_n(\theta^\prime)$ we have that $\Q_{pn}(\chi, \theta)=\Q_{pn}(\chi, \theta')=\Q_{pn}(\xi, \theta)=\Q_{pn}(\xi, \theta')$.
   We easily deduce that 
  $\Q_{pn}(\chi)=\Q_{pn}(\xi)$ using that $\Q(\theta) \sbs \Q(\chi)$ and $\Q(\theta^\prime) \sbs \Q(\xi)$.
  We want to apply Lemma \ref{triv}(ii). If $c(\xi)_p=1$ then $\Q(\xi)\sbs \Q_n$ and consequently $\Q(\chi)\sbs \Q_{pn}(\chi)\sbs \Q_{pn}$, but this
  is impossible as $c(\chi)_p\geq p^2$. Hence $c(\chi)_p\geq p$ and by Lemma \ref{triv}(ii) $c(\chi)_p=c(\xi)_p$. 
Recall that $H<G$. Write $k=|H|_{p'}.$ By induction,  $\Q_{c(\chi)_p}=\Q_{c(\xi)_p}\sbs \Q_{pk}(\xi)\sbs \Q_{pn}(\xi)=\Q_{pn}(\chi)$, 
then we are done.
   \end{proof}

    \section{Character Triples and Fields}
    
    If $\chi \in \irr G$ lies in a $p$-block $B$, we denote $h(\chi)$ the $p$-height of $\chi$ (we
    will sometimes just refer to $h(\chi)$ as the height of $\chi$).
    We remind the reader that if $N \sbs \ker\chi$, then the height of $\chi$ as a character of $G$
    and as a character of $G/N$ can be different.  For instance, the character of degree 2 of $\sf S_3$
    has 2-height zero, but as a character of $\sf S_4$ has 2-height 1.
    The next result clarifies this situation.
    
    \begin{lem}\label{r}
    Suppose that $\chi \in \irr B$, where $B$ is a $p$-block of $G$ with defect group $P$.
    Suppose that $K \sbs \ker\chi$ and let $\bar\chi \in \irr{G/K}$ the  character $\chi$ viewed
    as a character in $G/K$. Let $\bar B$ be the $p$-block of $G/K$ containing $\bar \chi$.
    \begin{enumerate}
   \item
   There is a defect group $\bar D$ of $\bar B$ such that $\bar D \le PK/K$.
   \item We have that
    $$p^{h(\chi)}={|PK/K| \over |\bar D|} p^{h(\bar\chi)}\, .$$
    In particular, $h(\chi)\ge h(\bar\chi)$ and if $h(\chi)=0$ then $PK/K=\bar D$ and $h(\bar\chi)=0$.   
    \item
    If $K\sbs \zent G$, then $PK/K=\bar D$ and $h(\chi)=h(\bar\chi)$.
    \end{enumerate} 
    \end{lem}
    
    \begin{proof}
    The first part is Theorem 9.9 of \cite{N2}. Since $\chi$ lies over $1_K$, it follows that $B$
    covers the principal block of $K$, by Theorem 9.2 of \cite{N2}. By Theorem 9.26 of \cite{N2},
    we have that $P\cap K \in \syl pK$, so $|K:K\cap P|_p=1$. 
    Now, 
    $$\chi(1)_p=\bar\chi(1)_p={|G/K|_p \over |PK/K|} {|PK/K|\over |\bar D|}p^{h(\bar \chi)}={|G|_p\over |P|}{|PK/K|\over |\bar D|}p^{h(\bar \chi)} \, ,$$
    and we use the definition of $h(\chi)$.
    The third part follows from Lemma 2.2 of \cite{R}.
    \end{proof}

%    \begin{lem}
%    Suppose that $\chi_N=\theta \in \irr N$, where $N \nor G$. Then
%    $DN/N \in \syl p{G/N}$, where
%    $D$ is a defect group of the block of $\chi$, and $h(\chi)=h(\theta)$.   \end{lem}
%    
%    \begin{proof}
%    By   Theorem 9.26 of \cite{N2}, we know that $D\cap N$ is a defect group of the defect of $\theta$.
%    By Lemma 2.2 of \cite{Mu}, we know that $h(\theta) \le h(\chi)$.
%    Now,  
%    $${|N|_p \over |D\cap N|}p^{h(\theta)}=\theta(1)_p=\chi(1)_p={|G|_p \over |D|} p^{h(\chi)} \, ,$$
%    and we deduce that 
%    $${|DN/N| \over |G/N|_p}=p^{h(\chi)-h(\theta)} \, .$$
%    This proves what we wanted.
%    \end{proof}

\begin{lem}\label{hct}
    Suppose that $G^*$ is a finite group, $N, Z \nor G^*$ such that $N\cap Z=1$, where $Z \sbs \zent{G^*}$.
    Let $N^*=N \times Z$. Let $\theta \in \irr N$ be $G$-invariant, and $\lambda \in \irr Z$.
    Let $ \theta^*=\theta \times 1_Z$, $ \lambda^*=1_N \times \lambda$,
    and assume that  $(\lambda^*)^{-1} \theta^*$ extends to some $\tau \in \irr{G^*}$. Then the map $\chi^* \mapsto \chi^*\tau$ defines a character triple isomorphism $(G^*, N^* , \lambda^*) \rightarrow  (G^*, N^*,  \theta^*)$. Let $\chi^* \in \irr{G^* | \lambda^*}$.
    If $\chi=\chi^*\tau$ has height zero in $G^*$, then $\chi^*$ has height zero in $G^*/N$.
    \end{lem}

    \begin{proof}      We have that $\tau_{N}=\theta$.  The fact that  $\chi^* \mapsto \chi^*\tau$ defines a character triple isomorphism follows 
    from Lemma 11.27 of \cite{I}.
    
  Let $\chi^*\in \irr{G^* |\lambda^*}$. Then $N\sbs \ker{\chi^*}$ and we can see $\chi^*$
    has a character of $G^*/N$ lying over $\lambda$ (identified with a character of $N^*/N$).
    Assume that $\chi=\chi^*\tau \in \irr{G^* |\theta^*}$ has height zero in $G^*$,  we want to prove that
    $\chi^*$ has height zero in $G^*/N$. 
    Let $B$ be the $p$-block of 
    $G^*/N$ that contains $\chi^*$,  and let $D^*/N$
    be a defect group of $B$. By Proposition 2.5(b) of \cite{NS}, $D^*/N$ is contained in $DN/N$,
    where $D$ is a defect group of $\chi$. 
%    Write $G^* = \hat G/N$, $N^* = \hat N/N$ and
%    $D^*=\hat D/N$.
%      Since $D^*$ is a defect group of $G^*$ and $N^*$ is central in $G^*$, we have that $|D^*N^*:D^*|_p=1$.
%    
    Since $\theta$ is an irreducible constituent of $\chi_N$ and $\chi$ has height zero, 
    by Theorem \ref{ht0}, we know that   $\theta$ has height zero. By   Theorem 9.26 of \cite{N2}, 
    we know that $D\cap N$ is a defect group of the block of $\theta$.
Therefore:
    $$\chi^*(1)_p=(\chi(1)/\theta(1))_p= |G^*:DN|_p \, .$$
    
  By definition,  $$\chi^*(1)_p=|G^*/N: D^*/N|_p p^{h(\chi^*)} \ge |G^*/N:DN/N|_pp^{h(\chi^*)} =
   |G^*:DN|_p p^{h(\chi^*)}\, .$$
We conclude that $p^{h(\chi^*)}=1$, as wanted.
   \end{proof}

    Next we use the theory of character triples, as developed in \cite[Chapter 11]{I}.
    Recall that if $G/N$ is a group, by \cite[Theorem 11.17]{I} there exists a finite central
    extension $(\Gamma, \pi)$ of $G/N$ such that $A=\ker\pi\cong M(G/N)$
    and the standard map $\eta \colon \irr A \to M(G/N)$ is an isomorphism. 
    In particular, by \cite[Theorem 11.19]{I}, if $G/N$ is perfect then $\Gamma$ is perfect.
    We will also make use of the results contained in \cite[Section 3]{GP}. 
    %We notice that the latest ArXiv version of \cite{GP}: arXiv:2202.13825v3 
    %corrects the published version. 

    \begin{thm}\label{ct}
    Suppose that $(G,N, \theta)$ is a character triple, where $\theta \in \irr N$     and $G/N$ is perfect.   Then there exists a character triple isomorphism
    $$(G,N,\theta) \rightarrow (\Gamma,A,\lambda), $$ where $\Gamma$ is perfect, $A=\zent{\Gamma}$,
    $\Q(\lambda)\sbs \Q(\theta)$ and $\Q(\chi)=\Q(\chi^*,\theta)$ for every $\chi \in \irr{G|\theta}$, where $\chi^*$ corresponds to $\chi$
    under the character triple isomorphism. 
%    Also, $c(\chi^*)$ divides $|G/N|c(\theta)$.
     Furthermore, if $\chi$ has height zero in $G$, then $\chi^*$ has height zero in $\Gamma$.
       \end{thm}

    \begin{proof}
    We consider a  {\em canonically constructed} character triple $(\Gamma, A, \lambda)$ isomorphic to $(G, N, \theta)$ in the sense of \cite{GP}. Notice that the values of any $\psi \in \irr{\Gamma|\lambda}$
    are in $\Q_{|G|}$. (See the paragraph before Corollary 3.3 of \cite{GP}.)
    By Theorem 3.6 of \cite{GP}, we have  that whenever 
    $\sigma \in {\rm Gal}( \Q_{|G|}/\Q(\theta))$, then $(\chi^*)^\sigma=(\chi^\sigma)^*$.
    Hence $\Q(\chi)=\Q(\chi,\theta)=\Q(\chi^*,\theta)$, as wanted.
    We do notice that $\Gamma$ is perfect using Theorem 11.19 of \cite{I}. 
    For the second part, we notice that
    the construction of the character triple isomorphism in \cite{GP} follows the construction in Theorem 11.28 of \cite{I}.
    We have that $(G, N, \theta)$ isomorphic to $(G^*,N^*, \theta^*)$ where $G^*\sbs G\times \Gamma$, $N^*=N\times A$, $A \sbs \zent{G^*}$ and $\theta^*=\theta\times 1_A$; in fact
    $G\cong G^*/A$.
    Also $(G^*, N^*, \theta^*)$ is isomorphic to $(G^*, N^*, \lambda^*)$ where $\lambda^*=1_N\times \lambda$ (here $\theta^*(\lambda^*)^{-1}$ extends to $\tau \in \irr{G^*}$).
    Finally $(G^*, N^*, \lambda^*)$ is isomorphic to $(\Gamma, A, \lambda)$ using that $\Gamma\cong G^*/N$.     
    Given $\chi \in \irr{G|\theta}$ of height zero in $G\cong G^*/A$, the first character triple isomorphism just sends $\chi$ to $\chi$ viewed as a character of $G^*$. By 
    Lemma \ref{r}(iii), we have that $\chi$ has height zero as a character of $G^*$. By Lemma \ref{hct} we have that $\chi=\chi^*\tau$, and $\chi^*\in \irr{G^*| \lambda^*}$ has height  zero in $G^*/N$. Since the last character triple isomorphism just sends $\chi^*$ to $\chi^*$ seen as a character of $\Gamma\cong G^*/N$, we have that $\chi^*$ has height zero    
    as a character of $\Gamma$, and the second part of the statement follows. \end{proof}
 
%We remark that an alternative of using \cite{GP} instead of Theorem  \ref{ct} in the proof of our main result,
%one can prove that in this case it is possible
%to  obtains a character triple isomorphism that preserves values of characters over $F=\Q_{|G|_{p'}}(\theta)$,
% using the ideas of the proof of Theorem 5.1 of \cite{NTT} together with Theorem \ref{gg} below. 
% This is weaker than the results in \cite{GP} 
%but enough to prove the results in this paper. 

%\medskip

We believe that the following result might be useful in the future.
(It can be used, together with the ideas of the proof of Theorem 5.1 of \cite{NTT}, as an alternative to Theorem \ref{ct} in the proof of our main result.) %
%\medskip
%
%
%The following theorem can be used to give an alternative proof of Theorem \ref{ct} in the case where $p=2$. We believe that it might be useful in the future. 

      \begin{thm}\label{gg}
    Suppose that $\chi \in \irr G$ has 2-height zero. Let $n=|G|_{2'}$, $F=\Q_n(\chi)$. Then $\chi$ can be afforded by an absolutely irreducible
     $F$-representation.
    \end{thm}
    
    \begin{proof}
    We want to show that $m_F(\chi)=1$.
    We know by Corollary 10.13 of \cite{I} that $m_F(\chi) \le 2$.  Suppose that $D$ is a defect group of the block
    of $\chi$, and let $H=\norm GD$. Let $C=\cent GD$ and $Z=\zent D$.
    By Lemma \ref{index2lem},
     we have that $\chi_H$ contains some $\psi \in \irr H$ with 2-height zero, $\Q(\psi)\sbs F$
    and $[\chi_H, \psi]$ is odd. If we can show that $\psi$ is afforded by
    an $F$-representation, then so it is $\psi^G$ and by  Corollary 10.2(c) of \cite{I}
   we have that $m_F(\chi)$ divides $[\psi^G, \chi]=[\psi,\chi_H]$ which is odd.
   
      We work thus to show that $\psi$ can be afforded by an $F$-representation.
      By Theorem \ref{cliff}, and arguing by induction on $|G|$, we may assume that $\psi$ is quasiprimitive. In particular, as $DC\nor H$, we have
       that $\psi_{DC}$ is homogeneous.  By Lemma \ref{qp} we have that $H/CD$ is an odd-order group, and the canonical character $\theta\in \irr{C/Z}$
       of the block of $\psi$ is $H$-invariant. 
        Then $\psi_{DC}=e \nu$ is an $F$-valued character and $\eta\in \irr {CD}$. Notice that by Corollary 11.29 of \cite{I} $e$ is odd. 
        Using again Lemma \ref{qp}, we have that $\eta=\theta_\lambda$, where $\lambda \in \irr{D}$ is linear.
    On the other hand, as $CD$ is a central product of $C$ and $D$, we have that $\eta=\nu \cdot \lambda$, where $\nu \in \irr{C|\mu}$ and $\lambda \in \irr{D|\mu}$ being $\mu=\lambda_Z$. In particular, $\eta_C=\nu$.
%     By Isaacs' restriction lemma (Lemma 6.8(d) of \cite{N3}), we have that
%     $\nu=\eta_C \in \irr C$.  

    We claim that $\eta$ can be afforded by an $F$-representation. Let $P$ be
    a Sylow $p$-subgroup of $C$. Let $x \in P-Z$. As $\nu=\eta_C=(\theta_\lambda)_C$, using the values of $\theta_\lambda$
    we have that
     $\nu(x)=0$. If $x \in Z$, then $\nu(x)=\theta(1)\lambda(x)$. Therefore
    $$\nu_P={\theta(1) \over |P:D|} \mu^P \, .$$
    Notice that ${\theta(1) \over |P:D|}$ is not divisible by 2 because $\theta \in \irr{C/Z}$ has
    2-defect zero.  Finally, since $\psi_{DC}=e \eta$, where $e$ is odd, then
    $\psi_C=e\nu$ and thus $[\psi_P, \mu^P]$ is odd. Hence $[\psi, \mu^H]$ is odd. 
    Since $\mu$ is afforded by an $F$-representation, we have that $m_F(\psi)$ is odd, by Corollary 10.2(c) of \cite{I}. 
    But Corollary 10.13 of \cite{I} implies that $m_F(\chi) \le 2$. Therefore $m_F(\psi)=1$, and this completes the proof.
    \end{proof}
    
 According to M. Geline, Theorem \ref{gg} can also be proved by using the main theorem of \cite{GG} and 
some number theoretical arguments. We notice that the case where $\chi$ has odd degree of Theorem \ref{gg} follows from a theorem of Fong
\cite[Corollary 10.13]{I} using \cite[Corollary 10.2(h)]{I}.  
    
      \section{Fields of Characters in $\mathcal F_p$}
     
     We briefly pause  in our journey to the proof of Theorem A and the reduction theorem for Conjecture D to show the easy containments in Theorem B and Conjecture D. In other words,
     we show that every number field in $\mathcal F_p$ is the field of values of some irreducible character of $p$-height zero.
     We will also show that the statement of Conjecture D is implied by the statement of \cite[Conjecture B]{N1}.

 \begin{thm}\label{G}
 Suppose that $\Q \sbs F\sbs \Q_{n}$,
where $n$ is the conductor of $F$. Write $n=p^a m$, where $m$ is not divisible by $p$.
If $|\Q_n:\langle F, \Q_m\rangle|$ is not divisible by $p$, then there is a solvable group $G$ and
$\chi \in \irr G$ of $p$-height zero such that $\Q(\chi)=F$.
\end{thm}

\begin{proof}
Let  $\zeta_n$ be a primitive $n$-th root of unity, and let
$C_n=\langle \zeta_n\rangle$ be the cyclic group of order $n$, which is
 acted on faithfully by $\mathcal G={\rm Gal}(\Q_n/\Q)$. Let $G$ be the semidirect product of
 $C_n$ with $H={\rm Gal}(\Q_n/F)\leq \mathcal G$. 
 If $\lambda \in \irr {C_n}$
 has order $n$, then $\lambda^G=\chi\in \irr G$ has field of values $F$.
 Let $\nu=\lambda_{C_m}$, where $C_m\leq C_n$ has order m. 
 Then $\nu$ has order $m$. 
 Notice that 
 $H_\nu={\rm Gal}(\Q_n/\langle F, \Q_m\rangle)$ has order not divisible by $p$ by hypothesis.
 Thus $G_\nu=C_nH_\nu$. By the Fong-Reynolds Theorem 9.14 of \cite{N2}, it follows that $\chi$
 has height zero if and only if $\lambda^{G_\nu}$ has height zero, which it has, because it has degree
 not divisible by $p$.
 \end{proof}

We can also prove one of the implications in Corollary C. 
 
\begin{cor}\label{onCorollaryC}
Suppose that $d \neq 1$ is an odd square-free integer.
Then there is a  group $G$ and a $2$-height zero character $\chi \in \irr G$ such that $\Q(\chi)=\Q(\sqrt d)$.
\end{cor}
 
\begin{proof}~~By considering the cyclic group of order 4, we may assume that $d\ne \pm 1$
is an odd square-free integer.
Let $F=\Q(\sqrt d)$.
If $d\equiv 1$ mod 4, then $F\sbs \Q_{|d|}$, $|d|$ is the conductor of $F$, and we are done by Theorem \ref{G}.
Suppose that $d \equiv 3$ mod 4. Let $n=2^2|d|$. By Theorem \ref{G},
we only need to show that $\langle \Q_{|d|}, \Q(\sqrt d) \rangle =\Q_n$.
Since $|\Q_n:\Q_{|d|}|=2$, this can only fail if $\Q(\sqrt d) \sbs \Q_{|d|}$, which is not possible
because the conductor of $\Q(\sqrt d)$ is $n$. 
\end{proof}

 We finish this section by showing that Conjecture D follows from the Alperin-McKay-Navarro conjecture \cite[Conjecture B]{N1}.
 We recall that the Alperin-McKay-Navarro conjecture predicts, for a $p$-block $B$ of a finite group $G$ and its Brauer first main correspondent $b$, that there is a bijection 
 between the sets of height zero characters of $B$ and $b$ such that, if $\chi$ corresponds to $\chi^*$ then $\Q_n(\chi)=\Q_n(\chi^*)$
 where $n=|G|_{p'}$. We care to mention that the currently accepted and most studied form of \cite[Conjecture C]{NT} is more general, predicting the
 existence of an $\mathcal H$-equivariant bijection between height zero characters of $B$ and $b$, where the Galois group $\mathcal H$, which contains ${\rm Gal}(\Q_{|G|}/\Q_n)$, is defined in \cite[Section 2]{N1}. When we write {\em the Alperin-McKay-Navarro conjecture} we refer to the statement predicting $\mathcal H$-equivariant character bijections.

\begin{thm}\label{cons}
Conjecture D follows from the Alperin-McKay-Navarro conjecture.
\end{thm}

\begin{proof}
Let $\chi \in \irr B$ be of $p$-height zero, where $B$ has defect group $D$. By the Alperin-McKay-Navarro conjecture, there is 
$\tau \in \irr{\norm GD}$ of heigh zero, in the Brauer correspondent block $b$ of $\norm GD$
such that $\Q_n(\chi)=\Q_n(\tau)$, where $n=|G|_{p'}$. In particular $c(\chi)_p=c(\tau)_p$
reasoning as in Lemma \ref{triv}.  
Hence, it is enough to prove that $Q_{c(\tau)_p}\sbs \Q_{pm}(\tau)$ with $m=|\norm G D|_{p'}$. 
We may then assume that $D \nor G$. 
In this case, we can apply Lemma \ref{normald}.  
\end{proof}

    \section{The Reduction}
    
    There is one more issue that we have to solve before proving a reduction to quasi-simple groups
    of Conjecture D. If $G/N$ is a $p'$-group, $\theta \in \irr N$ is $G$-invariant and $p$-rational,
    it is not necessarily true that the characters of $G$ over $\theta$ are $p$-rational (even if they extend
    $\theta$ and $p$ is odd),  as shown by $p=3$ and ${\tt SmallGroup}(24,4)$.
    We need the following.
    
    \begin{thm}\label{cp1}
    Suppose that $G/N$ is a simple group of order coprime to $p$. Let $\theta \in \irr N$ be $G$-invariant and $p$-rational.
    If $\chi \in \irr{G|\theta}$, then $c(\chi)_p \le p$.
    \end{thm}
    
    \begin{proof}
    Suppose first that $G/N$ has order $q$. Then this follows from Theorem B of \cite{V}.
      Suppose now that $G/N$ is perfect.  By Theorem \ref{ct}, there is an isomorphic character triple 
    $(G^*,N^*, \theta^*)$ such that $N^*=\zent{G^*}$, $\Q(\theta^*) \sbs \Q(\theta)$ and $\Q(\chi)=\Q(\chi^*,\theta^*)$
    whenever $\chi \in \irr{G|\theta}$ and $\chi^*$ corresponds to $\chi$ under the character triple isomorphism. In particular, we have that $G^*/N^*$ is a simple non abelian $p'$-group, and $\theta^*$ is $p$-rational. 
    Let $x \in G^*$, and let $\delta \in \irr{N^*\langle x\rangle}$ be over $\theta^*$. Since $N^*\langle x \rangle /N^*$ is cyclic and $\theta^*$ is $G^*$-invariant,  $\delta_{N^*}=\theta^*$.
    Since $x_p \in N^*$ and $\delta$ is linear, we have that $\delta(x)=\delta(x_p)\delta(x_{p'})=\theta^*(x_p)\delta(x_{p'})
     \in \Q_{|G^*|_{p'}}$. Then every $\delta \in \irr{N^*\langle x \rangle | \theta^*}$ is $p$-rational. Since $\chi_{N^*\langle x\rangle}$ is a sum of irreducible characters lying over $\theta^*$ and the choice of $x \in G^*$ was arbitrary, we are done. 
%    Write $\theta^*=\nu \mu$, where $o(\nu)$ is a $p$-power and $o(\mu)$ is not divisible by $p$. Since $\nu$ is a power of $p$,
%    we have that $\nu$ is rational valued, so $o(\nu)=2$. 
    \end{proof}

    \begin{conj}\label{conjqs}
    Let $\chi \in \irr G$ of $p$-height zero, where $G$
    is a quasi-simple group.
   Assume in addition that the $p$-block $B$ containing $\chi$ is not (virtual) Morita equivalent over an absolutely unramified complete discrete valuation ring to a $p$-block of any group
   $H$ with $|H:\zent H|<|G:\zent G|$.
   Write $c(\chi)=p^a m$. Then $\Q_{p^a} \sbs \Q_{pm}(\chi)$.
   \end{conj}
      
    Notice that if the $p$-block $B$ of $\chi$ is virtual Morita
    equivalent over an absolutely unramified complete discrete valuation ring to a $p$-block
    $\tilde B$, by Theorem 1.6 of \cite{KL},  there is a height zero character $\tilde \chi$ in $\tilde B$ and an integer $l$ not divisible by $p$
    such that $\Q_l(\chi)=\Q_l(\tilde\chi)$. By Lemma \ref{triv}, we have that $c(\chi)_p=c(\tilde \chi)_p=p^a$,
    and therefore $\Q_{p^a} \sbs \Q_{pl}(\chi)$ if and only if $\Q_{p^a} \sbs \Q_{pl}(\tilde\chi)$.
    Hence, if $c(\chi)_{p'}=m$ and $c(\tilde\chi)_{p'}=m_1$, then
    $\Q_{p^a} \sbs \Q_{pm}(\chi)$ if, and only if, $\Q_{p^a} \sbs \Q_{pm_1}(\tilde\chi)$, by Lemma \ref{useful}.
    We will use this argument below.
    
    \begin{thm}\label{reduction}
    Let $G$ be a finite group, and let $p$ be a prime.
     Let $\chi \in \irr G$ be of $p$-height zero, and write $c(\chi)=p^am$, where $a\ge 0$ and $p$ does not divide $m$.      If Conjecture \ref{conjqs} is true, then $\Q_{p^a} \sbs \Q_{pm}(\chi)$.
       \end{thm}
    
\begin{proof}
We argue by induction on $|G:\zent G|$. We may assume that $a\ge 2$,
because otherwise the statement is trivially satisfied. 
 Let $n=|G|_{p^\prime}$. By Lemma \ref{useful}, it is enough to show that 
$\Q_{p^a} \sbs \Q_{pn}(\chi)$.

\smallskip
\noindent
{\sl Step 1. If $N\nor G$ is a proper normal subgroup, then we may assume that $\chi_N=e\theta$ 
for some $\theta \in \irr N$, with $c(\theta)_p<c(\chi)_p$.}
\smallskip

Let $T$ be the stabilizer of $\theta$ in $G$, let $\psi \in \irr{T|\theta}$ be the Clifford correspondent of $\chi$ over $\theta$. 
By Theorem \ref{cliff}, we have that $\Q_{pn}(\psi)=\Q_{pn}(\chi)$ and $c(\chi)_p=p^a=c(\psi)_p$. 
Assume that $T<G$. Then $|T:\zent T|<|G:\zent G|$, and by induction, if $m_1$ is the $p'$-part of the conductor of $\psi$, we have that $\Q_{p^a} \sbs \Q_{pm_1}(\psi) \sbs \Q_{pn}(\psi)=\Q_{pn}(\chi)$.
By Lemma \ref{useful}, we deduce that $\Q_{p^a} \sbs \Q_{pm}(\chi)$. Hence,
we may assume that  $\chi_N=e\theta$.  In particular, $\Q(\theta) \sbs \Q(\chi)$, and $c(\theta)$ divides $c(\chi)$.
If $c(\theta)_p=c(\chi)_p$, and $m_2$ is the
$p'$-part of the conductor of $\theta$, then by induction $\Q_{p^a} \sbs \Q_{pm_2}(\theta) \sbs \Q_{pm_2}(\chi)$.
Again by Lemma \ref{useful}, we deduce what we want.

\smallskip
\noindent
{\sl Step 2. $G$ does not have a normal subgroup of index $p$.}
\smallskip

Otherwise, let $N$ be a normal subgroup of $G$ with $G/N$ of order $p$.
By Step 1, $\chi_N$ is homogeneous. Since $G/N$ is cyclic we can write $\chi_N=\theta \in \irr N$.

    Suppose that $G=NH$, $M=N\cap H$ and there exists  some $H$-invariant $\varphi \in \irr{M}$
     of $p$-height zero such that
    $[\theta_{M}, \varphi]$ is not divisible by $p$ and $\Q_{pn}(\varphi) \sbs \Q_{pn}(\theta)$.  
    Let $\xi \in \irr H$ be an extension of $\varphi$. By Lemma \ref{elem}, we have  that $\xi$ has $p$-height zero. 
        By Lemma \ref{mult}(c), we have that  $\Q_p(\xi, \theta)=\Q_p(\chi,\varphi)$.  
        Notice that
    $\Q_{pn}(\chi) \sbs \Q_{pn}(\chi,\varphi) \sbs \Q_{pn}(\chi,\theta)=\Q_{pn}(\chi)$.
    Therefore $\Q_{pn}(\chi)= \Q_{pn}(\chi,\varphi)=\Q_{pn}(\theta, \xi)$. Also, $\Q_{pn}(\xi) \sbs \Q_{pn}(\chi)$.    
    Write $c(\theta)_p=p^b$ and $c(\xi)_p=p^c$. 
    We have that $\Q(\xi) \sbs \Q_{pn}(\xi) \sbs \Q_{pn}(\chi) \sbs \Q_{p^an}$. Therefore $c \le a$.
    Notice that if $\theta$ and $\xi$ are $p$-rational, then 
    $\Q(\chi) \sbs \Q_{pn}(\chi)= \Q_{pn}(\chi,\varphi)=\Q_{pn}(\theta, \xi)=\Q_{pn}$, but we are assuming that
    $a\ge 2$.  Hence $b, c \ge 1$.  Now $\Q_{pn}(\theta) \sbs \langle \Q_{p^bn}, \Q_p\rangle$,
    and $\Q_{pn}(\xi) \sbs \langle \Q_{p^cn}, \Q_p\rangle$.
      Thus $$\Q(\chi) \sbs \Q_{pn}(\chi)= \langle \Q_{pn}(\theta), \Q_{pn}(\xi)\rangle \sbs
      \langle \Q_{p^bn},\Q_{p^cn},  \Q_p\rangle \sbs \Q_{p^d n}  \, ,$$
      where $d={\rm max}(b,c)$. Hence $p^a \le p^d$.  Since $b< a$ by Step 1 and $c\le a$,
      we conclude that $d=c=a$.
    If $|H:\zent H|<|G:\zent G|$ then, by induction, we have that  $\Q_{p^a} \sbs \Q_{pn}(\xi) \sbs \Q_{pn}(\chi)$, and we are done (using Lemma \ref{useful}). 
    
    By Lemma \ref{index2lem}, we may assume that the defect group $D_0$ of the block of $\theta$
    is normal in $G$. By Lemma \ref{extnormald}, we are done in this case.

\smallskip
\noindent
{\sl Step 3. $G$ does not have a proper normal subgroup of index not divisible by $p$.}
\smallskip

Otherwise, let $N\nor G$ such that $G/N$ is simple of order not divisible by $p$. By Step 1, $\chi_N=e\theta$. Recall that 
$c(\chi)_p\ge p^2$.  Hence, by Theorem \ref{cp1},  we have that $c(\theta)_p\ge p$, in this case.
 By  
Lemma 4.2.(ii) of \cite{NT}, we conclude that $c(\chi)_p=c(\theta)_p$,
contradicting Step 1.

%\smallskip
%\noindent
%{\sl  Step 4.}~~{\sl The block $B$ containing $\chi$ is not (virtual) Morita equivalent over an absolutely unramified complete discrete valuation ring to a $p$-block of any group $H$ with
% $|H:\zent H|<|G:\zent G|$.}
%  \smallskip
%  
%  This follows from Theorem 1.6 of \cite{KL}.
  
\smallskip
\noindent
{\sl Final Step.}
Let $N$ be a maximal normal subgroup of $G$. By Steps 2 and 3, $G/N$ is simple non-abelian of order divisible by $p$. 
By Step 1, write
$\chi_N=e\theta$, where $c(\theta)_p=p^b$ and $b<a$. 
By Theorem \ref{ct},  there is a quasi-simple group $G^*$
and a $p$-height zero character $\chi^*$ of $G^*$ such that $\Q(\chi)=\Q(\chi^*,\theta)$. Write $c(\chi^*)=p^c k$,
where $k$ is not divisible by $p$.
Then the conductor of the field $\Q(\chi^*,\theta)=\langle \Q(\chi^*), \Q(\theta)\rangle$ is the least common multiple of the conductors of $\chi^*$
and $\theta$.
In particular, its $p$-part is $p^{{\rm max}(c,b)}$. As $c(\chi)=c(\Q(\chi^*,\theta))$, we have that
$a={\rm max}(c,b)$. Since $b<a$, we have that $a=c$.
Thus $c(\chi)_p=c(\chi^*)_p$.  If the $p$-block $B^*$ of $\chi^*$ is virtual Morita
    equivalent over an absolutely unramified complete discrete valuation ring to a $p$-block
    $\tilde B$ of a group $H$ with $|H:\zent H|<|G:\zent G|$, then $c(\chi^*)_p=c(\tilde \chi)_p$
    for some $\tilde \chi \in \tilde B$. As explained before the statement of this theorem,
    we would be done in this case, using induction. Therefore, by Conjecture \ref{conjqs},
we have that  $\Q_{p^a} \sbs \Q_{pk}(\chi^*) \sbs \Q_{pk}(\chi)$.   By using Lemma \ref{useful},
the proof of the theorem follows. \end{proof}

     \section{Theorem A for quasisimple groups}
     The aim of this section is to prove Conjecture \ref{conjqs} in the case that $p=2$, see
     Theorem \ref{qs} below.
     In the light of Theorem \ref{reduction} (and Theorem \ref{G}), this will complete the proof of Theorem A (and Theorem B). 
     
     \medskip
       The following is useful when working on extensions
     of quasi-simple groups.
     
     \begin{thm}\label{normal subgroup}
   Suppose that $G/N$ is abelian. Write $n=|G|_{2'}$. Suppose that $\chi \in \irr G$ has 2-height zero and suppose that
    $\Q_{2^a} \sbs \Q_n(\chi)$, where  $c(\chi)_2=2^a$.
     Let $\theta \in \irr N$ be an irreducible constituent of $\chi_N$, and write $c(\theta)_2=2^b$.
     Then $\Q_{2^b} \sbs \Q_n(\theta)$.
        \end{thm}

     \begin{proof}
     We argue by induction on $|G:N|$. We may assume that $G/N$ has prime index.
     By Theorem \ref{cliff}, we may assume that $\chi_N=\theta$ is irreducible.
     By Lemma 4.2.(ii) of \cite{NT}, we may assume that
     $G/N$ has order 2.  Let $D$ be a defect group of the block of $\chi$, such that $DN=G$
     and $D_0=D\cap N$ is a defect group of the block of $\theta$. 
     Let $H=\norm G{D_0}$ and $M=H\cap N$. By Lemma \ref{index2lem},      there is $\varphi \in \irr M$ of height zero, $\Q_n(\varphi) \sbs \Q_n(\theta)$ with an extension $\xi \in \irr H$
     such that $\Q(\chi,\varphi)=\Q(\theta,\xi)$. 
     Thus $\Q_n(\chi,\varphi)=\Q_n(\theta,\xi)$. Hence $\Q_{2^a} \sbs \Q_n(\theta,\xi)$. We may assume that $b \ge 2$.
     Suppose that $c(\xi)_2=2^c$. We know that $\Q_{2^c} \sbs \Q_n(\xi)$ by 
     Lemma \ref{extnormald}. If $\xi$ is 2-rational, then we are done.
     So we may assume that $c\ge 2$, and thus $i \in \Q_n(\xi)$. Then $\Q_n(i) \sbs \Q_n(\xi) \cap \Q_n(\theta)$. Since  $\Q_n(\xi), \Q_n(\theta) \sbs \Q_{|G|}$, and ${\rm Gal}(\Q_{|G|}/\Q_n(i))$ is a cyclic 2-group, we then have that  $\Q_n(\xi)\sbs \Q_n(\theta)$ or $\Q_n(\theta) \sbs \Q_n(\xi)$.
     Since $\Q_{2^a} \sbs \Q_n(\theta,\xi)$, we may assume the second.
     Then $\Q_{2^a} \sbs \Q_n(\xi)$.   Thus $a \le c$.  If $H<G$, then by induction $\Q_{2^c} \sbs \Q_n(\varphi) \sbs \Q_n(\theta)$. Thus we may assume that $D_0\nor G$.
     But in this case, we are done by Lemma \ref{normald}.
     \end{proof}
     
      \begin{thm}\label{qs}
    Let $\chi \in \irr G$ of 2-height zero, where $G$
    is a quasi-simple group.  
     Assume in addition that the $2$-block $B$ containing $\chi$ is not (virtual) Morita equivalent over an absolutely unramified complete discrete valuation ring to a $2$-block of any group $H$ with
   $|H:\zent H|<|G:\zent G|$.
    If $c(\chi)=2^am$, where $m$ is odd, then $\Q_{2^a} \sbs \Q_m(\chi)$. 
    \end{thm}

\begin{thm}\label{simple2}
Theorem \ref{qs} is true in the case $G/\ZB(G)$ is a simple group of Lie type in characteristic $2$.
\end{thm}    
     
\begin{proof}
In the case $S:=G/\ZB(G)$ is isomorphic to $\AAA_5$, $\AAA_6$, $\AAA_8$, $\SL_3(2)$, $\SU_4(2)$, $\Sp_6(2)$,
$\PSL_3(4)$, $\PSU_6(2)$, $\Omega^+_8(2)$, $\tw2 B_2(8)$, $G_2(4)$, $F_4(2)$, $\tw2 F_4(2)'$, or $\tw2 E_6(2)$, the statement
is checked using \cite{GAP}.
% \edit{except for $2^2 \cdot \Omega^+_8(2)$ and $(2^2 \times 3) \cdot \tw2 E_6(2)$}. 
Hence we may assume that $S$ is not isomorphic to any of these simple groups.
This implies that $G$ is a quotient (by a central subgroup) of $\GC^F$, where $\GC$ is a simple, simply connected, algebraic 
group in characteristic $2$ and $F:\GC\to\GC$ a Steinberg endomorphism. It follows from the main result of \cite{Hum} that
any $2$-block $B$ of $G$ is either of defect $0$, or of maximal defect. Moreover, in the former case $\chi \in \Irr(B)$ is just the 
Steinberg character; in particular it is rational and so we are done. In the latter case, $\chi(1)$ is odd, and the statement follows
from \cite[Theorem A1]{NT}.
\end{proof} 

\begin{thm}\label{alt}
Theorem \ref{qs} is true in the case $G/\ZB(G)$ is an alternating or sporadic simple group.
\end{thm}    
     
\begin{proof}
In the case $S:=G/\ZB(G)$ is isomorphic to $\AAA_n$ with $5 \leq n \leq 8$ or one of $26$ sporadic simple groups, the statement
is checked using \cite{GAP}.
% \edit{except for the Monster}. 
Hence we may assume that $S = \AAA_n$ with $n \geq 9$.

\smallskip
(a) First we consider the case $G=S$. If $\chi$ extends to $\SSS_n$ then $\chi$ is rational. Otherwise, \cite[Theorem 2.5.13]{JK} shows
that the $\SSS_n$-character lying above $\chi$ is labeled by a self-associated partition of $n$, with hook lengths 
along the main diagonal of the Young diagram being the $k \geq 1$ odd integers $2h_1+1 > 2h_2+1 > \ldots > 2h_k+1$, in which 
case the only possible irrational values of $\chi$ are 
$$\frac{(-1)^{(n-k)/2} \pm \sqrt{(-1)^{(n-k)/2}\prod^k_{i=1}(2h_i+1)}}{2}.$$
Since $n = \sum^k_{i=1}(2h_i+1)$, we see that $(-1)^{(n-k)/2} = 1$ if and only if $\prod^k_{i=1}(2h_i+1) \equiv 1 \pmod{4}$.
It follows that $\chi$ is $2$-rational, and we are done in this case.

\smallskip
(b) It remains to handle the case $G = 2\AAA_n$. We change the notation, and let $\tilde B$ the $2$-block of $G$ containing $\chi$.
We also embed $G$ in a double cover $\tilde G = 2\SSS_n$ of $\SSS_n$. By \cite[Lemma 2.2]{De}, $\tilde B$ contains a unique 
block $B$ of $G/\ZB(G) \cong \AAA_n$. Similarly, by \cite[Lemma 2.1]{De}, $\tilde B$ is covered by a unique block 
$\tilde B_s$ of $\tilde G$, and $B$ is covered by a unique block $B_S$ of $(\tilde G)/\ZB(G) \cong \SSS_n$. All these blocks 
$\tilde B$, $\tilde B_S$, $B$, and $B_S$ have the same weight $w \geq 0$.

If $\chi$ is trivial on $\ZB(G)$, then we are done by (a). Hence we may assume that $\chi$ is a spin character of $G$. Since $\tilde B$ 
contains a spin character of height zero, by \cite[Proposition 3.1]{De} we must have that $w \in \{0,1\}$ (and $\chi$ is the unique 
spin character of height zero in $\tilde B$). The defect groups $D$ of $B$ are the Sylow $2$-subgroups of $\AAA_{2w}$, and hence
$D=1$. This implies by \cite[Lemma 2.2]{De} that the defect groups $\tilde D$ of $\tilde B$ are precisely $\ZB(G)$. Thus we are 
in the case of central defect, and $\chi$ has relative defect zero, with respect to the faithful linear character $\mu$ of $\ZB(G)$.
Now \cite[Theorem 2.1]{N0} gives an explicit bijection between defect zero characters of $G/\ZB(G)$ and relative defect zero characters
of $G$. In particular, this bijection implies that $\chi(g) = 0$ whenever the $2$-part $g_2$ of $g$ is not in $\ZB(G)$. 

We now show that $\chi$ is $2$-rational. Indeed, write $g = g_2g_{2'}$ as the product of the $2$-part and the $2'$-part of $g$.
By the above, $\chi(g)=0$ if $g_2 \notin \ZB(G)$. If $g_2 \in \ZB(G)$, then $\chi(g) = \mu(g_2)\chi(g_{2'}) = \pm \chi(g_{2'})$ is 
$2$-rational. Hence the statement follows in this case as well.
\end{proof}    
     
\begin{thm}\label{clas}
Theorem \ref{qs} is true in the case $G/\ZB(G)$ is a simple classical group in odd characteristic.
\end{thm}    
     
\begin{proof}
In the case $S:=G/\ZB(G)$ is isomorphic to $\PSU_4(3)$, $\PSU_6(2)$, $\Omega_7(3)$, or $G_2(3)$, the statement
is checked using \cite{GAP}. (Note that in Theorem \ref{qs} we are dealing with a $2$-block $B$ of $G$, and so we may assume 
${\bf O}_{2'}({\bf Z}(G))$ is cyclic. Hence in the case of covers of $\PSU_4(3)$, it suffices to handle 
the two covers $12_1 \cdot \PSU_4(3)$ and $12_2 \cdot \PSU_4(3)$, which are given in \cite{GAP}.)

Hence we may assume that $S$ is not isomorphic to any of these simple groups, as well 
as any Lie type group in characteristic $2$.
This implies that $G$ is a quotient (by a central subgroup) of $\GC^F$, where $\GC$ is a simple, simply connected, algebraic 
group in odd characteristic $r \neq 2$ and $F:\GC\to\GC$ a Steinberg endomorphism. Without any loss, we may 
replace $G$ by $\GC^F$. Let $(\GCD,\FD)$ be dual to $(\GC,F)$; in particular, $\GCD$ is of adjoint type,
and let $\GD:= (\GCD)^{\FD}$.

By the main result of \cite{BM}, the set $\Irr(B)$ of complex characters in the $2$-block $B$ containing $\chi$ is contained
in $\EC_2(G,s)$ for some semisimple element $s \in \GD$ of odd order. Suppose that $s$ is not quasi-isolated (in the sense of
\cite{Bon}). Then, by the main result of \cite{BR}, $B$ is Morita equivalent to a $2$-block of a group $H$ with
$|H:\ZB(H)| < |G:\ZB(G)|$. Moreover, by \cite[Proposition 4.2]{FK} this Morita equivalence descends to an absolutely unramified discrete valuation ring contrary to our assumption.
%%one could use: of a group with all composition factors of order less than $|S|$. Then it would apply to $\CB_{\GD}(s)$. 
Hence we may assume that $s$ is quasi-isolated.

Assume in addition that $\GC$ is not of type $A$. By the classification result of Bonnaf\'e \cite[Table 2]{Bon}, the odd-order assumption
on $s$ implies that $s=1$. In this case, by \cite[Theorem 21.14]{CE}, $\EC_2(G,s)$ is just the set of irreducible characters in
the principal $2$-block $B_0$ of $G$. In such a case, $\chi(1)$ is odd, and the statement follows
from \cite[Theorem A1]{NT}.

It remains to consider the case $\GC$ is of type $A$. The same arguments as in the preceding paragraph allow us to assume
that $s \neq 1$, and so $s$ is {\bf not} isolated, see \cite[Table 2]{Bon}. The main result of \cite{BDR} together with \cite[Proposition 4.2]{FK} now shows 
that $B$ is again Morita equivalent over an absolutely unramified complete discrete valuation ring to a $2$-block of a group $H$
with $|H:\ZB(H)| < |G:\ZB(G)|$,  contrary to our assumption.
\end{proof}     
     
%{The proof of Theorem \ref{clas} shows that it remains to handle the {\bf isolated} blocks of exceptional groups in odd characteristics.} 
          
\begin{thm}\label{exc}
	Theorem \ref{qs} is true in the case $G/\ZB(G)$ is a simple exceptional group in odd characteristic.
\end{thm}              
          
\begin{proof}
	We keep the same notation from Theorem \ref{clas}. That is $G$ is a (quotient by a central subgroup) of $\GC^F$, where $\GC$ is a simple, simply connected, algebraic group in odd characteristic $r \neq 2$ and $F: \GC \to \GC$ is a Steinberg endomorphism. Arguing as in Theorem \ref{clas}, we can also assume that $s$ is isolated. By Lemma \ref{useful}, we observe that is suffices to prove that $\Q_{c(\psi)_2} \subset \Q_{|G|_{2'}}(\psi)$ for every height zero character $\psi \in \Irr_0(B)$. We may also assume $c(\psi)_2 \geq 4$ since otherwise $\Q_{c(\psi)_2}=\Q$ and the statement is trivally true.
	
	\smallskip
	(a) Let us first assume that 
	%		$\ZB(\GC )$ is connected and
	the defect group of $B$ has order $|G^\ast: \CB_{G^\ast}(s)|_2$ and $\CB_{\GC^\ast}^\circ(s)$ has only components of classical type. Then $\psi \in \mathcal{E}(G,st)$ for some element $t \in G^\ast$ which is $2$-central in the group $H:=\CB^\circ_{\GC^\ast}(s)^{F^\ast}$, see \cite[(2.1)]{MalleHZ}.
	
	If $\ZB(\GC) \neq 1$ then we let $\GC \lhd \tilde{\GC}$ a regular embedding with dual surjective morphism $\iota^\ast: \tilde{\GC}^\ast \to \GC^\ast$. Otherwise set $\tilde{\GC}:=\GC$. There exists a semisimple element $\tilde{s} \in \tilde{G}^\ast:=(\tilde{\GC}^\ast)^{F^\ast}$ of $2'$-order such that $\iota^\ast(\tilde{s})=s$ and $\tilde{t} \in \CB_{\tilde{G}^\ast}(\tilde{s})_2$ with $\iota^\ast(\tilde{t})=t$. We let $\chi \in \mathcal{E}(\tilde{G},\tilde{s} \tilde{t})$ be a character covering $\psi$. By \cite[Theorem 4.7.9]{GeckMalle}, \cite[Proposition 4.5.5]{GeckMalle} and using that $\CB_{\tilde{G}^\ast}(\tilde{s} \tilde{t})$ has only components of classical type we have $\Q(\chi) \subset \Q_{\mathrm{o}(\tilde{s} \tilde{t})}$ and so $c(\chi)_2 \leq \mathrm{o}(\tilde{t})$.
	%	The arguments from before work as well in this case and we have again a bijection $\mathcal{E}_2(\tilde{G},\tilde{s}) \to \mathcal{E}_2(\tilde{G}(\tilde{s}),\tilde{s})$ and the blocks in $\mathcal{E}_2(\tilde{G},\tilde{s})$ are all of maximal defect.
	
	%
	%Now $\chi \in \mathcal{E}(\tilde{G},\tilde{s}\tilde{t})$ with $\tilde{t}$ a $2$-central element in $\tilde{H}:=\C^\circ_{\mathbf{H}^\ast}(\tilde{s})$.
	%Let $\chi \in \Irr(\tilde{G})$ covering $\psi$. There exists some element  such that $\chi \in \mathcal{E}(\tilde{G},\tilde{s} \tilde{t})$.
	\smallskip
	(a1)
	Assume now first that $|\chi(1):\psi(1)|_2>1$. Then $G=E_7(q)$ and $\chi=(\psi')^{\tilde{G}}$ for some $\psi' \in \Irr(G \ZB(\tilde{G}) \mid \psi)$. In this case, $c(\chi)_2=c(\psi')_2$ by Theorem \ref{cliff} and $\Q(\chi) \subset \Q(\psi')$. Since $\psi \in \Irr(G)$ has height zero, it follows that $\psi_{\ZB(G)}$ is trivial by \cite[Lemma 8.7]{Expectional}. Hence, we can choose $\chi \in \Irr(\tilde{G} \mid \psi)$ with the additional property that $\chi$ is trivial on $\ZB(\tilde{G})$. A consequence of this choice is that $c(\psi')=c(\psi)$ and $\Q(\psi')=\Q(\psi)$.
	
	Since $\chi$ is the unique character in its $\Irr(\tilde{G}/G)$-orbit which is trivial on $\ZB(\tilde{G})$, it follows that any Galois automorphism that stabilizes the $\Irr(\tilde{G}/G)$-orbit of $\chi$ also stabilizes $\chi$.
	% By \cite[Proposition 2.7]{Classical} we have $\mathrm{o}(\tilde{t})=\mathrm{o}(t)$ or $\mathrm{o}(\tilde{t})=2 \mathrm{o}(t)$ so that $\tilde{t}^{\mathrm{o}(t)}=z$, where $z$ is the unique involution in $\Z(\tilde{G}^\ast)$.
	Hence, \cite[Theorem 4.7.9]{GeckMalle} and \cite[Proposition 4.5.5]{GeckMalle} show therefore that $\Q(\chi) \subset \Q_{\mathrm{o}(\tilde{s}) \mathrm{o}(t)}$.
	%since $\chi \otimes \hat{z}=\chi$, where $\hat{z}$ is the unique character associated to $z$ as in \cite[Proposition 2.5.20]{GeckMalle}.
	
	Recall that $\psi \in \mathcal{E}(G,st)$ has height zero and the defect group of $B$ has order $|G^\ast:\CB_{G^\ast}(s)|_2$. Assume that $e$ is an integer coprime to $\mathrm{o}(t)$ such that $t$ is $H$-conjugate to $t^e$. We claim that $t=t^e$. Arguing as in \cite[Theorem 5.9]{MalleHZ}, we see that $t$ lies in the centralizer of a Sylow $d$-torus $\mathcal{S}_d$ of $\CB^\circ_{\GC^\ast}(s)$ for $d$ the order of $q$ modulo $4$. By the proof of \cite[Corollary 2.4]{BroughRuhstorfer}, the Sylow $2$-subgroup $W_2$ of the Weyl group of $W:=\N_{H}(\mathcal{S}_d)/\CB_{H}(\mathcal{S}_d)$ is self-normalizing. Moreover, since $t$ is $2$-central its centralizer $W(t)$ in $W$ contains a Sylow $2$-subgroup $W_2$ of $W$.
	
	Since $\N_{H}(\mathcal{S}_d)$ controlls $H$-fusion in $\CB_{H}(\mathcal{S}_d)$ by \cite[Proposition 5.11]{MalleHZ} it follows that $t$ and $t^e$ are conjugate by an element $w \in \N_{W}(W(t))$. In particular, by conjugacy of Sylow subgroups in $W(t)$ we can assume that $w \in \N_W(W_2)=W_2$ and so $t=t^e$ as claimed. This implies that $\Q_{\mathrm{o}(t)} \subset \Q_{o(s)}(\psi)$ by \cite[Proposition 3.3.15]{GeckMalle} and thus $c(\psi)_2 \geq \mathrm{o}(t)_2$ by Lemma \ref{useful}. On the other hand, $$c(\chi)_2=c(\psi)_2 \geq \mathrm{o}(t)\geq c(\chi)_2,$$
	and so $\mathrm{o}(t)_2=c(\psi)_2$. Hence, $\Q_{c(\psi)_2} \subset \Q_{o(s)}(\psi)$.
	
	\smallskip (a2)
	Assume now that $|\chi(1):\psi(1)|_2=1$. In this case, $\tilde{t}$ is $2$-central in $\CB_{\tilde{G}^\ast}(\tilde{s})$. The argument from the first case now show that $\Q_{c(\chi)_2} \subset \Q_{\mathrm{o}(\tilde{t})_2} \subset \Q_{o(\tilde{s})}(\chi)$. Hence, the claim follows in this case from Theorem \ref{normal subgroup}.

	% for every prime $s$ we have that $\chi_{\ZB(\tilde{G})_{s}}$ is trivial whenever $\psi_{\ZB(G)_s}$ is trivial.  
	
	%In particular, if $\tilde{t}^2=1$ then, $c(\chi)_2=1$.

\smallskip
(b) Suppose now that $s=1$, i.e. that $B$ is a unipotent block. We can assume that $B$ has non-maximal defect since otherwise the statement follows from \cite[Theorem A.1]{NT}. By Lemma \ref{normald} we can also assume that $B$ has non-central defect. In this case, $B$ is one of the blocks considered in \cite[Lemma 7.1]{Expectional}. In case (i) of \cite[Lemma 7.1]{Expectional}, the defect group of $B$ is dihedral, so every character in $\Irr_0(B)$ is $2$-rational by \cite[Theorem 8.1]{Sambale}. Hence, the claim holds. In case (ii), $G=E_8(q)$ and the height zero characters where explicitly described in \cite[Lemma 7.4]{Expectional}. It follows from this description that $\Irr_0(B) \subset \cup_{t}\mathcal{E}(G,t)$, where $t \in G^\ast$ runs over elements with $t^2=1$, and all height zero characters of $\Irr_0(B)$ have $\ZB(G)$ in their kernel. Now \cite[Theorem 4.7.9]{GeckMalle} and \cite[Proposition 4.5.5]{GeckMalle} show that $\Q(\chi)$ is a cyclotomic field or $\Q(\chi) \subset\Q(\sqrt{r})$ which in both cases implies the statement.
	
	%	where $\Q_{\tilde{t}} \subset \Q_{\mathrm{o}(\tilde{t}) }$ is the subfield defined as in Schaeffer-Fry, Taylor. 
	%	
	%	There exists a semisimple element $\tilde{t} \in \tilde{G}^\ast$ such that $\iota^\ast(\tilde{t})=t$ and $\chi \in \mathcal{E}(\tilde{G},\tilde{t})$. Denote by $(\phi,\tilde{t})$ the Jordan correspondent of $\chi$.
	%, the character $\chi^\sigma \in \mathcal{E}(\tilde{G},\tilde{t}^k)$ corresponds to the character $(\phi^\sigma,\tilde{s}^k)$. Since $q=r^f$ for some integer $f$ and $\Q(\sqrt{r}) \subset \Q_r$, we deduce that $c(\chi)_2=\mathrm{o}(\tilde{t})_2$.
\smallskip
(c)
An analysis of the tables in \cite{KessarMalle} shows that in the remaining casese $G=E_8(q)$ and $\CB_{\GC^\ast}(s)$ is of type $E_6 A_2$.
	%Let us now assume that $1 \neq s$ so that $G=E_8(q)$ by \cite[Lemma 7.1]{Expectional}.
	%If $B$ is a block associated to a Levi subgroup of type $E_7(q)$, then $\Irr_0(B) \subset \mathcal{E}(G,st)$ with $t^2=1$. Therefore, $\Q(\chi) \subset \mathbb{Q}_{\mathrm{o}(s) r}$ for all $\chi \in \Irr_0(B)$.
	%We can therefore assume that $B$ is a block associated to a Levi subgroup of type $E_6$.
	Let $G(s)$ be the $F$-fixed points of the connected reductive group in duality with $\CB_{\GC^\ast}(s)$. For $t \in \CB_{G^\ast}(s)_2$ let $\psi_{G,st}: \mathcal{E}(G,st) \to \mathcal{E}(\CB_{G^\ast}(st),1)$ be Digne--Michel's unique Jordan decomposition for groups with connected center as in in \cite[Theorem 4.7.1]{GeckMalle}. Moreover, note that the centralizer of $st$ in $\CB_{\GC^\ast}(s)$ is connected. Hence, there exists a bijection $\psi_{G(s),st}: \mathcal{E}(G(s),st) \to \mathcal{E}(\CB_{G^\ast}(st),1)$ which can be uniquely determined from Digne--Michel's unique Jordan decomposition in a regular embedding of $G(s)$. We have a bijection $\mathcal{J}:\mathcal{E}_2(G,s) \to \mathcal{E}_2(G(s),s)$ which is the union of the bijections $\psi_{G(s),st}^{-1} \circ \psi_{G,st}$ with $t \in \CB_{G^\ast}(s)_2$, see \cite[Lemma 2.3]{Expectional}. By \cite[Theorem 4.7.9]{GeckMalle} and the construction of $\mathcal{J}$ it follows that $\mathcal{J}$ is $\mathrm{Gal}(\Q_{|G|}/\Q_{\mathrm{o}(s)})$-equivariant.
	%Indeed, if $\chi \in \mathcal{E}(G,st)$ and $\phi=\psi_{G,st}(\chi)$. By construction, $\psi_{G(s),st}(\mathcal{J}(\chi))=\phi$ and therefore, $\psi_{G(s),st^k}(\mathcal{J}(\chi^\sigma))=\phi^\sigma=\psi_{G(s),st^k}(\mathcal{J}(\chi))^\sigma$.
	Moreover, by \cite[Proposition D]{Expectional} there exists a bijection $c \mapsto b$ between blocks contained in $\mathcal{E}_2(G(s),s)$ and the blocks contained in $\mathcal{E}_2(G,s)$ such that $\mathcal{J}(\Irr_0(b))=\Irr_0(c)$. From this it follows that $\Q_{o(s)}( \mathcal{J}(\chi))=\Q_{o(s)}(\chi)$ and $c(\chi)_2=c(\mathcal{J}(\chi))_2$. Hence, it suffices to consider the height zero characters of the unipotent blocks of $G(s)$. However, by \cite[Theorem 17.7]{CE} the unipotent blocks of $G(s)$ are isomorphic in a natural way to the blocks of $G(s)_{\mathrm{sc}}$, the $F$-fixed points of the simply connected covering of $G(s)$. By what we have established about unipotent blocks, $\Q_{c(\mathcal{J}(\chi))_2} \subset \Q_{|G|_{2'}}(\mathcal{J}(\chi)_2)$. Therefore, $\Q_{c(\chi)_2} \subset \Q_{|G|_{2'}}(\psi)$ for all height zero characters $\psi$.
\end{proof}

 We finish this paper proving another  corollary of Theorem A.
 For an integer $e\ge 1$,
 let $\sigma_e$ be the Galois automorphism in ${\rm Gal}(\Q^{\rm ab}/\Q)$ fixing $2'$-roots of unity and sending $\xi$ to $\xi^{1+ 2^e}$, where $e\ge 1$.
 If $G$ is any finite group, then
 $$\mathcal G={\rm Gal}(\Q_{|G|}/\Q_{|G|_{2'}})=\langle \tau_1, \tau_2\rangle,$$
 where $\tau_i$ is the restriction of $\sigma_i$ to $\Q_{|G|}$.
 Notice that a character $\chi \in \irr G$ is 2-rational if, and only if, $\chi$ is $\mathcal G$-fixed.   
The set of 2-height zero characters fixed under the action of $\langle \sigma_1\rangle$ has been recently studied 
in connection with the number of generators of 2-defect groups (see  \cite{RSV, NRSV, V1}).

We have the following.

\begin{thm}\label{1}
Let $\chi \in \irr G$ of 2-height zero. Then $\chi$ is 2-rational if, and only if, $\chi$ is $\sigma_1$-fixed.
\end{thm}

\begin{proof}
 Let $m=|G|_{2'}$. Suppose that $\chi$ is $\sigma_1$-fixed.  
 Then $\Q_m(\chi)$ is also fixed by $\sigma_1$.
If $\chi$ is not 2-rational, then $i \in \Q_m(\chi)$ by Theorem A. However, $\sigma_1(i)=i^3 \ne i$, a contradiction. 
\end{proof}

Theorem \ref{1} is not true for characters which do not have height zero.
The smallest example is an irreducible character $\chi$ of degree 2 of a semidihedral group of order
16 with field of values  $\Q(\chi)=\Q(\sqrt{-2})$.

\end{document}